\documentclass[12pt]{amsart}
\usepackage{fancybox}
\usepackage{shuffle}
\usepackage{amsmath}
\usepackage{amsfonts}
\usepackage{amssymb}
\usepackage[all]{xy}           
\usepackage{bm}
\usepackage{bbding}
\usepackage{txfonts}
\usepackage{amscd}
\usepackage{appendix}

\usepackage{xspace}
\usepackage[shortlabels]{enumitem}
\usepackage{ifpdf}

\ifpdf
  \usepackage[colorlinks,final,backref=page,hyperindex]{hyperref}
\else
  \usepackage[colorlinks,final,backref=page,hyperindex,hypertex]{hyperref}
\fi
\usepackage{tikz}
\usepackage[active]{srcltx}
\usepackage{extarrows}

\makeatletter

\newtheorem{thm}{Theorem}[section]
\newtheorem{lem}[thm]{Lemma}
\newtheorem{cor}[thm]{Corollary}
\newtheorem{pro}[thm]{Proposition}
\newtheorem{ex}[thm]{Example}
\theoremstyle{definition}
\newtheorem{rmk}[thm]{Remark}
\newtheorem{defi}[thm]{Definition}
\newtheorem{thm*}{Theorem}

\newcommand{\nc}{\newcommand}
\newcommand{\delete}[1]{}

\nc{\mlabel}[1]{\label{#1}}  
\nc{\mcite}[1]{\cite{#1}}  
\nc{\mref}[1]{\ref{#1}}  
\nc{\mbibitem}[1]{\bibitem{#1}} 

\delete{
\nc{\mlabel}[1]{\label{#1}{\hfill \hspace{1cm}{\bf{{\ }\hfill(#1)}}}}
\nc{\mcite}[1]{\cite{#1}{{\em{{\ }(#1)}}}}  
\nc{\mref}[1]{\ref{#1}{{\em{{\ }(#1)}}}}  
\nc{\mbibitem}[1]{\bibitem[\em #1]{#1}} 
}

\newcommand {\emptycomment}[1]{}

\nc{\oprn}{\theta}
\nc{\Oprn}{\Theta}

\nc{\calo}{\mathcal{O}}
\nc{\oop}{$\mathcal{O}$-operator\xspace}
\nc{\oops}{$\mathcal{O}$-operators\xspace}
\nc{\mrho}{{\bm{\varrho}}}
\nc{\emk}{\mathbf{K}}
\nc{\invlim}{\displaystyle{\lim_{\longleftarrow}}\,}
\nc{\ot}{\otimes}

\newcommand{\lon }{\,\rightarrow\,}
\newcommand{\be }{\begin{equation}}
\newcommand{\ee }{\end{equation}}

\newcommand{\g}{\mathfrak g}
\newcommand{\h}{\mathfrak h}

\newcommand{\huaB}{\mathcal{B}}
\newcommand{\huaS}{\mathcal{S}}

\newcommand{\huaL}{\mathcal{L}}
\newcommand{\huaR}{\mathcal{R}}

\newcommand{\huaG}{\mathcal{G}}

\newcommand{\huaO}{{\mathcal{O}}}

\newcommand{\frkg}{\mathfrak g}

\newcommand{\frkL}{\mathfrak L}

\newcommand{\frkR}{\mathfrak R}

\newcommand{\frkX}{\mathfrak X}

\newcommand{\half}{\frac{1}{2}}

\newcommand{\pair}[1]{\left\langle #1\right\rangle}

\newcommand{\Courant}[1]{\left\llbracket  #1\right\rrbracket }

\newcommand{\Id}{{\rm{Id}}}

\newcommand{\br}[1]{   [ \cdot,    \cdot  ]   }

\newcommand{\Hom}{\mathrm{Hom}}

\newcommand{\Ad}{\mathrm{Ad}}

\newcommand{\gl}{\mathfrak {gl}}

\newcommand{\ad}{\mathrm{ad}}
\newcommand{\pr}{\mathrm{pr}}

\newcommand{\de}{\mathrm{d}}
\newcommand{\frkad}{\mathfrak{ad}}

\nc{\CV}{\mathbf{C}}

\newcommand{\LYA}{Lie-Yamaguti algebra}
\newcommand{\ltp}{\Courant{\cdot,\cdot,\cdot}}

\newcommand{\xrightleftarrow}{\mathrel{\raise{4mm}}}
\begin{document}

\title[The classical Lie-Yamaguti Yang-Baxter equation and Lie-Yamaguti bialgebras]{The classical Lie-Yamaguti Yang-Baxter equation and Lie-Yamaguti bialgebras}

\author{Jia Zhao}
\address{Jia Zhao, School of Sciences, Nantong University, Nantong 226019, Jiangsu, China}
\email{zhaojia@ntu.edu.cn}

\author{Yu Qiao*}
\address{Yu Qiao (corresponding author), School of Mathematics and Statistics, Shaanxi Normal University, Xi'an 710119, Shaanxi, China}
\email{yqiao@snnu.edu.cn}

\date{\today}

\begin{abstract}
In this paper, we develop the bialgebra theory for \LYA s. For this purpose, we exploit two types of compatibility conditions: local cocycle condition and double construction. We define the classical Yang-Baxter equation in \LYA s and show that a solution to the classical Yang-Baxter equation corresponds to a relative Rota-Baxter operator with respect to the coadjoint representation. Furthermore, we generalize some results by Bai in \cite{BCM3} and Semonov-Tian-Shansky in \cite{STS} to the context of \LYA s.
Then we introduce the notion of matched pairs of Lie-Yamaguti algebras, which leads us to the concept of double construction Lie-Yamaguti bialgebras following the Manin triple approach to Lie bialgebras. We prove that matched pairs, Manin triples of \LYA s, and double construction Lie-Yamaguti bialgebras are equivalent. Finally, we clarify that a local cocycle condition is a special case of a double construction
for Lie-Yamaguti bialgebras.
\end{abstract}

\thanks{{\em Mathematics Subject Classification} (2020): 17B38, 16T25, 17A30, 17B81}

\keywords{Lie-Yamaguti bialgebra, the classical Yang-Baxter equation in \LYA s, relative Rota-Baxter operator, Manin triple, matched pair}

\maketitle

\vspace{-1.1cm}

\tableofcontents

\allowdisplaybreaks

\section{Introduction}
Roughly speaking, a bialgebra structure on a given algebra $\g$ is endowed with a compatible coalgebra structure on $\g$. For instance, a Lie bialgebra is a Lie algebra $(\g,[\cdot,\cdot])$ together with a cobracket $\delta:\g\longrightarrow\otimes^2\g$ such that $\delta^*:\otimes^2\g^*\longrightarrow\g^*$ is also a Lie algebra structure on $\g^*$ and a certain compatibility condition is satisfied. As we all know, compatibility conditions for a Lie bialgebra can be expressed as three aspects: derivation condition, cocycle condition, and double construction. A Lie bialgebra enjoys an elegant property that these conditions are equivalent and that every condition has its own advantage. More precisely, since the corresponding exterior algebra $\wedge^\bullet\g$ is in fact a graded Lie algebra, the cobracket $\delta$ can be seen as a derivation on $\wedge^\bullet\g$. Thus the derivation condition reads that
$$\delta[x,y]=[\delta(x),y]+[x,\delta(y)],\quad \forall x,y\in \g.$$
The  notation $[x,\delta(y)]$ means that $(\ad_x\otimes{\Id}+{\Id}\otimes\ad_x)\delta(y)$, where $\ad:\g\longrightarrow\gl(\g)$ is the adjoint representation of $\g$. The cocycle condition can be read that the cobracket $\delta$ is a $1$-cocycle on $\g$ with coefficients in the tensor representation $\ad\otimes{\Id}+{\Id}\otimes\ad$, for we are able to form the tensor representation of two representations. Finally, the double construction is that there is a Lie algebra structure on $\g\oplus\g^*$ together with a nondegenerate and symmetric bilinear form.

Parallel to Lie bialgebras, there exist three types of compatibility conditions for $3$-Lie bialgebras:
derivation condition, cocycle condition, and double construction.
Several works such as \cite{B.G.S,BCLM,BGLZ} were devoted to bialgebra theory for $3$-Lie algebras, or generally $n$-Lie algebras.
The derivation condition was referred in  \cite{BCLM}. However, it is unknown whether there is a $3$-Lie algebra structure on $\wedge^\bullet\g$, thus Bai, Guo, and Sheng investigated cocycle conditions and double constructions as compatibility conditions in \cite{B.G.S}.
Since there is no tensor representation on $3$-Lie algebras, cocycle condition does not fit to study $3$-Lie bialgebras. However, for a given $3$-Lie algebra $(\g,[\cdot,\cdot,\cdot])$, they found that $(\otimes^3\g;\ad\otimes1\otimes1)$, $(\otimes^3\g;1\otimes\ad\otimes1)$, and $(\otimes^3\g;1\otimes1\otimes\ad)$ are representations of $\g$, where $\ad: \wedge^2\g \longrightarrow\gl(\g)$ is the adjoint representation, thus used local cocycle condition as its compatibility condition.  Moreover, Manin triples and matched pairs of $3$-Lie algebras were defined, which leads to the notion of double construction $3$-Lie bialgebras. Nevertheless, local cocycle condition and double construction are {\em not} equivalent any more, and an $r$-matrix for a $3$-Lie algebra gives rise to a local cocycle $3$-Lie bialgebra structure. Later, Sheng and his collaborators  found that an $r$-matrix, as a relative Rota-Baxter operator, gives rise to a twilled $3$-Lie algebra, while a twilled $3$-Lie algebra is not equivalent to a matched pair of $3$-Lie algebra in \cite{HST}. This is why an $r$-matrix does not give rise to a double construction $3$-Lie bialgebra structure.

To build bialgebra theory on other algebraic structure, many authors have made efforts in recent years. For example, Sheng and Tang used quadratic Leibniz algebras to study Leibniz bialgebras in \cite{T.S2}, where a quadratic Leibniz algebra is just the Manin  triple of Leibniz algebras. The compatibility condition for Leibniz bialgebras is also the double construction. Moreover, bialgebra theory and the classical Yang-Baxter equation for Hom-Lie algebra version were established in \cite{Yau}. Recently, Rota-Baxter Lie bialgebras and endo Lie bialgebras were studied in \cite{BGLM,BGS1,LS}. Chen, Sti$\acute{\textrm e}$non, and Xu examined weak Lie $2$-bialgebras by using big brackets with respect to which $\huaS^\bullet(V[2]\oplus V^*[1])$ is a graded Lie algebra, and proved that (strict) Lie 2-bialgebras
are in one-one correspondence with crossed modules of Lie bialgebras (\cite{chenz}). 
Moreover, they proved that there is a one-to-one correspondence between connected,
simply-connected (quasi-)Poisson Lie 2-groups and (quasi-)Lie 2-bialgebras in \cite{CSX}.
Later Lang, the corresponding author, and Yin proved that Lie 2-bialgebroids
are in one-one correspondence with crossed modules of Lie bialgebroids in \cite{LQY}). More importantly, Tang, Bai, Guo, and Sheng exploited linear deformations of the skew-symmetric classical $r$-matrices and their corresponding triangular Lie bialgebras in \cite{TBGS}, when studying cohomology and deformations of relative Rota-Baxter operators (also called $\huaO$-operators) on Lie algebras.

The notion of Lie triple algebras, or general Lie triple systems, which is a generalization of Lie algebras and Lie triple systems was introduced by Yamaguti in \cite{Yamaguti1}. Afterwards, Yamaguti gave the notion of representations and established cohomology theory of this object in \cite{Yamaguti2,Yamaguti3} during 1950's to 1960's. Later until earlier 21st century, Kinyon and Weinstein named this object as a \LYA ~in \cite{Weinstein} formally.  This kind of algebraic structures has attracted much attention recently. For instance, Benito and his colleagues investigated \LYA s related to simple Lie algebras of type $G_2$ \cite{B.D.E} and afterwards, they explored orthogonal and irreducible \LYA s in \cite{B.B.M} and \cite{B.E.M1,B.E.M2} respectively. Sheng and the first author focused on linear deformations, product structures and complex structures on \LYA s in \cite{Sheng Zhao} and later, relative Rota-Baxter operators and pre-\LYA s were introduced in \cite{SZ1}. Besides, we studied cohomology and deformations of relative Rota-Baxter operators on \LYA s in \cite{Zhao Qiao}.

Due to the importance of bialgebras and \LYA s, it is natural to develop a bialgebra theory for \LYA s. Motivated by Lie bialgebras and $3$-Lie bialgebras, one considers to define a Lie-Yamaguti bialgebra structure on a \LYA ~$(\g,\br,,\ltp)$ as a pair of two cobrackets $(\delta,\omega)$, where $\delta:\g\longrightarrow\otimes^2\g$ and $\omega:\g\longrightarrow\otimes^3\g$, such that one of the following compatibility conditions is satisfied:
\begin{itemize}
\item derivation condition: the cobrackets $\delta$ and $\omega$ is a derivation on $\wedge^\bullet\g$ with respect to the binary and ternary brackets respectively, i.e.,
\begin{eqnarray*}
\delta([x,y])&=&[\delta(x),y]+[x,\delta(y)],\\
\omega(\Courant{x,y,z})&=&\Courant{\omega(x),y,z}+\Courant{x,\omega(y),z}+\Courant{x,y,\omega(z)},\quad\forall x,y,z\in \g;
\end{eqnarray*}
\item cocycle condition: the cobrackets $\delta$ and $\omega$ are $1$-cocycles of $\g$ with respect to a certain representation;
\item double construction: there is a \LYA ~structure on $\g\oplus\g^*$ together with a symmetric, nondegenerate bilinear form.
\end{itemize}

Since we have not found a suitable \LYA ~structure on the exterior algebra $\wedge^\bullet\g$ so far, the derivation condition is not considered in this paper. Therefore firstly we investigate the cocycle condition in Section 3 after a preparation in Section 2. Since there is {\em no} natural tensor representation of a \LYA ~$\g$, so we decided to use the {\em local cocycle condition} as the compatibility condition parallel to that of $3$-Lie bialgebras in \cite{B.G.S}.  Namely, we observe that $(\otimes^2\g;{\Id}\otimes\ad,{\Id}\otimes\huaR),~(\otimes^2\g;\ad\otimes{\Id},\huaR\otimes{\Id})$ and $(\otimes^3\g;\ad\otimes{\Id}\otimes{\Id},\huaR\otimes{\Id}\otimes{\Id}),~(\otimes^3\g;{\Id}\otimes\ad\otimes{\Id},{\Id}\otimes\huaR\otimes{\Id}),
~(\otimes^3\g;{\Id}\otimes{\Id}\otimes\ad,{\Id}\otimes{\Id}\otimes\huaR)$ are representations of a \LYA ~$\g$, where $(\g;\ad,\huaR)$ is the adjoint representation of $\g$, thus we modify the cocycle condition as follows (Definition \ref{local}):
\begin{itemize}
\item $\delta_1$ is a $1$-cocycle with respect to the representation $(\otimes^2\g;{\Id}\otimes\ad,{\Id}\otimes\huaR)$;
\item $\delta_2$ is a $1$-cocycle with respect to the representation $(\otimes^2\g;\ad\otimes{\Id},\huaR\otimes{\Id})$;
\item $\omega_1$ is a $1$-cocycle with respect to the representation $(\otimes^3\g;\ad\otimes{\Id}\otimes{\Id},\huaR\otimes{\Id}\otimes{\Id})$;
\item $\omega_2$ is a $1$-cocycle with respect to the representation $(\otimes^3\g;{\Id}\otimes\ad\otimes{\Id},{\Id}\otimes\huaR\otimes{\Id})$;
\item $\omega_3$ is a $1$-cocycle with respect to the representation $(\otimes^3\g;{\Id}\otimes{\Id}\otimes\ad,{\Id}\otimes{\Id}\otimes\huaR)$,
\end{itemize}
where $\delta=\delta_1+\delta_2$ and $\omega=\omega_1+\omega_2+\omega_3$ are cobrackets on $\g$. Moreover, we define the classical Yang-Baxter equation in \LYA s, but its solution {\em fails} to give rise to a local cocycle Lie-Yamaguti bialgebra structure. However, we find that a solution to the classical Yang-Baxter equation is one-to-one correspondence to a relative Rota-Baxter operator with respect to the coadjoint representation. That is, we have the following theorem.

\begin{thm*}(Theorem \ref{THM:CYBE})
A skew-symmetric $2$-tensor $r\in \otimes^2\g$ is a solution to the classical Lie-Yamaguti Yang-Baxter equation if and only if the induced map $r^\sharp:\g^*\longrightarrow\g$ is a relative Rota-Baxter operator with respect to the coadjoint representation, where $\pair{r^\sharp(\xi),\eta}=\pair{r,\xi\otimes\eta}$, for all $\xi,\eta\in \g^*$.
 \end{thm*}

\noindent{Furthermore, we generalize some results in \cite{BCM3} and in \cite{STS} by Bai and Semonov-Tian-Shansky respectively to the context of \LYA s.}

In Section 4, motivated by the double of a Lie bialgebra, it is natural to consider the {\em double construction} as a compatibility condition for a Lie-Yamaguti bialgebra.
\emptycomment{
\begin{itemize}
\item there is a \LYA ~structure on $\g\oplus\g^*$ together with a symmetric, nondegenerate bilinear form.
\end{itemize}}
In order to extend this approach to the context of \LYA s, we introduce the notions of Manin triples and matched pairs of \LYA s. Moreover, we prove that matched pairs, Manin triples of \LYA s, and double construction Lie-Yamaguti bialgebras are equivalent. That is the following vital theorem.
\begin{thm*}{\rm(Theorem ~\ref{main})}
Let $(\g,\br,,\ltp)$ be a \LYA, and $\delta:\g\longrightarrow\otimes^2\g$ and $\omega:\g\longrightarrow\otimes^3\g$ linear maps. Suppose that a pair of structure maps $(\delta^*,\omega^*)$ defines a \LYA ~structure on $\g^*$. Then the following statements are equivalent:
\begin{itemize}
\item[(1)] $(\g,\g^*)$ is a double construction Lie-Yamaguti bialgebra;
\item[(2)] the quadruple $\Big(\g,\g^*;(\ad^*,-\huaR^*\tau),(\mathfrak{ad}^*,-\frkR^*\tau)\Big)$ is a matched pair of \LYA s, where $(\ad^*,-\huaR^*\tau)$ and $(\mathfrak{ad}^*,-\frkR^*\tau)$ are the coadjoint representations of $\g$ and $\g^*$ on $\g^*$ and $\g$ respectively;
\item[(3)] the triple $\Big(\g\oplus\g^*,\g,\g^*\Big)$ is a Manin triple of \LYA s.
\end{itemize}
 \end{thm*}

\noindent{Similar to the case of $3$-Lie bialgebras, local cocycle condition and double construction for a Lie-Yamaguti bialgebra are {\em not} equivalent as compatibility conditions. In fact, a local cocycle condition is a special case of double construction, which implies that properties of ternary operations on \LYA s or $3$-Lie algebras are quite different from those of binary operations on Lie algebras.

As a summary, all the relations among those concepts in the context of \LYA s are illustrated in the following diagram.

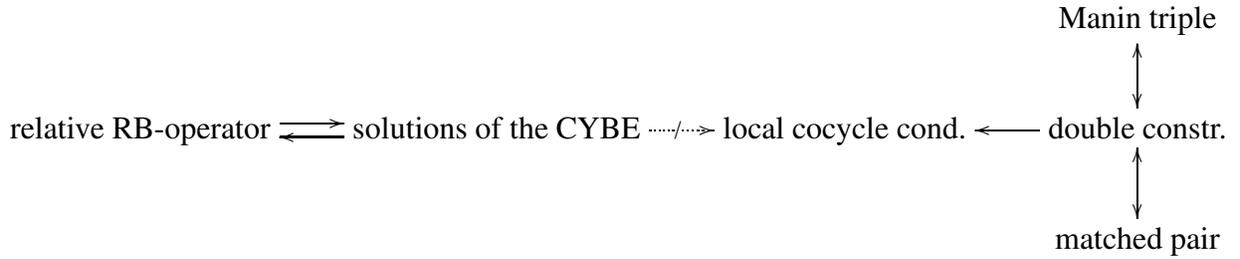
\[
\xymatrix{
                                      &                                      &                      & \text{Manin triple}\ar[d]  \\
 \text{relative RB-operator}\ar@<2.5pt>[r]  &\text{solutions of the CYBE}\ar@<2.5pt>[l]\ar@{.>}[r]|<<<<{/}&\text{local cocycle cond.} & \text{double constr.} \ar[l]\ar[u]\ar[d]  \\
                                  &                                      &                      &  \text{matched pair}\ar[u]
}
\]
\emptycomment{
The paper is organized as follows. In Section 2, we recall some preliminaries such as definitions of \LYA s and representations, and their cohomology theory. 
In Section 3, we define the classical Yang-Baxter equation in \LYA s ~and local cocycle Lie-Yamaguti bialgebras, and moreover clarify relationship between solutions to the classical Yang-Baxter equation and relative Rota-Baxter operators.
Finally, we introduce the notions of matched pairs, Manin triples of \LYA s, and double construction Lie-Yamaguti bialgebras, and establish the equivalence among standard Manin triple of $\g$ and $\g^*$, the corresponding matched pair, and double construction Lie-Yamaguti bialgebra $(\g,\g^*)$. Moreover, we show that a double construction Lie-Yamaguti bialgebra gives rise to a local cocycle one and several examples of double construction Lie-Yamaguti bialgebras are given in Section 4.}

Note once again that \LYA s  are a generalization of Lie algebras and Lie triple systems, thus when the given \LYA s in the present paper are restricted to the context of Lie triple systems, all the notions and conclusions are still valid.

\smallskip
{\bf Terminologies and Notations}: Let $\g$ be a vector space. For any $n$-tensor $T=x_1\otimes\cdots\otimes x_n\in \otimes^n\g~(n\geqslant 2)$ and $1\leqslant i<j\leqslant n$, define the switching operator to be
$$\sigma_{ij}(T)=x_1\otimes\cdots\otimes x_j\otimes\cdots \otimes x_i\otimes\cdots\otimes x_n.$$

In particular, for any $2$-tensor $x\otimes y\in \otimes^2\g$, the switching operator $\sigma_{12}$ is also denoted by $\tau$ in this article, i.e.,
$$\tau(x\otimes y)=y\otimes x.$$

In the tensor notation, we denote the Identity map $\Id$ by $1$ in this paper. For example, the tensor $\ad\ot\Id$ is denoted by $\ad\ot1$.

\smallskip
{\bf Acknowledgements}: We would like to thank Professor Yunhe Sheng and Rong Tang for their fruitful discussions and useful suggestions. Qiao was partially supported by NSFC grant 11971282.

\smallskip
\section{Preliminaries}
All vector spaces occurring in the article are assumed to be over a field of characteristic zero and finite-dimensional.
In this section, we briefly recall some basic notions such as Lie-Yamaguti algebras, representations and their cohomology theory. In particular, the coadjoint representation of a \LYA ~is a vital object in this paper.
\begin{defi}\cite{Weinstein}\label{LY}
A {\bf Lie-Yamaguti algebra} is a vector space $\g$ together with a bilinear bracket $[\cdot,\cdot]:\wedge^2  \mathfrak{g} \to \mathfrak{g} $ and a trilinear bracket $\Courant{\cdot,\cdot,\cdot}:\wedge^2\g \otimes  \mathfrak{g} \to \mathfrak{g} $, such that the following conditions hold
\begin{eqnarray*}
~ &&\label{LY1}[[x,y],z]+[[y,z],x]+[[z,x],y]+\Courant{x,y,z}+\Courant{y,z,x}+\Courant{z,x,y}=0,\\
~ &&\Courant{[x,y],z,w}+\Courant{[y,z],x,w}+\Courant{[z,x],y,w}=0,\\
~ &&\label{LY3}\Courant{x,y,[z,w]}=[\Courant{x,y,z},w]+[z,\Courant{x,y,w}],\\
~ &&\Courant{x,y,\Courant{z,w,t}}=\Courant{\Courant{x,y,z},w,t}+\Courant{z,\Courant{x,y,w},t}+\Courant{z,w,\Courant{x,y,t}},\label{fundamental}
\end{eqnarray*}
for all $x,y,z,w,t \in \g$. We denote a Lie-Yamaguti algebra by $(\g,\br,,\ltp)$.
\end{defi}

Note that a Lie-Yamaguti algebra $(\g,\br,,\ltp)$ with $[x,y]=0$ for all $x,y\in \g$ reduces to a Lie triple system, while with $\Courant{x,y,z}=0$ for all $x,y,z\in \g$ it reduces to a Lie algebra. 
\emptycomment{
The following example shows that there is a natural trilinear bracket on a Lie algebra, together with which it appears to be a Lie-Yamaguti algebra.

\begin{ex}
 Let $(\frkg,[\cdot,\cdot])$ be a Lie algebra. We define $\Courant{\cdot,\cdot,\cdot
 }:\wedge^2\g\otimes \g\lon \g$ to be  $$\Courant{x,y,z}:=[[x,y],z],\quad \forall x,y, z \in \mathfrak{g}.$$  Then $(\g,[\cdot,\cdot],\Courant{\cdot,\cdot,\cdot})$ forms a Lie-Yamaguti algebra.
\end{ex}}

The following example is taken from \cite{Nomizu}.

\begin{ex}
Let $M$ be a closed manifold \footnote{a smooth compact manifold without boundary} with an affine connection, and denote by $\frkX(M)$ the set of vector fields on $M$. For all $x,y,z\in \frkX(M)$, set
\begin{eqnarray*}
[x,y]&=&-T(x,y),\\
\Courant{x,y,z}&=&-R(x,y)z,
\end{eqnarray*}
where $T$ and $R$ are torsion tensor and curvature tensor respectively. It turns out that the triple
$ (\frkX(M),[\cdot,\cdot],\Courant{\cdot,\cdot,\cdot})$ forms an (infinite-dimensional) \LYA.
\end{ex}
\emptycomment{
\begin{defi}\cite{Takahashi}
Let $(\g,[\cdot,\cdot]_{\g},\Courant{\cdot,\cdot,\cdot}_{\g})$ and $(\h,[\cdot,\cdot]_{\h},\Courant{\cdot,\cdot,\cdot}_{\h})$ be two Lie-Yamaguti algebras. A {\bf homomorphism} from $(\g,[\cdot,\cdot]_{\g},\Courant{\cdot,\cdot,\cdot}_{\g})$ to $(\h,[\cdot,\cdot]_{\h},\Courant{\cdot,\cdot,\cdot}_{\h})$ is a linear map $\phi:\g \to \h$ preserving the Lie-Yamaguti algebra structure, i.e., for all $x,y,z\in \g,$ the following equalities hold
\begin{eqnarray*}
\phi([x,y]_{\g})&=&[\phi(x),\phi(y)]_{\h},\\
~ \phi(\Courant{x,y,z}_{\g})&=&\Courant{\phi(x),\phi(y),\phi(z)}_{\h}.
\end{eqnarray*}
\end{defi}}

The notion of representations of Lie-Yamaguti algebras was introduced in \cite{Yamaguti2}.
\begin{defi}\label{defi:representation}
Let $(\g,[\cdot,\cdot],\Courant{\cdot,\cdot,\cdot})$ be a Lie-Yamaguti algebra. A {\bf representation} of $\g$ is a vector space $V$ endowed with a linear map $\rho:\g \to \gl(V)$ and a bilinear map $\mu:\otimes^2 \g \to \gl(V)$, which satisfies the following conditions for all $x,y,z,w \in \g$,
\begin{eqnarray*}
~&&\label{RLYb}\mu([x,y],z)-\mu(x,z)\rho(y)+\mu(y,z)\rho(x)=0,\\
~&&\label{RLYd}\mu(x,[y,z])-\rho(y)\mu(x,z)+\rho(z)\mu(x,y)=0,\\
~&&\label{RLYe}\rho(\Courant{x,y,z})=[D_{\rho,\mu}(x,y),\rho(z)],\\
~&&\label{RYT4}\mu(z,w)\mu(x,y)-\mu(y,w)\mu(x,z)-\mu(x,\Courant{y,z,w})+D_{\rho,\mu}(y,z)\mu(x,w)=0,\\
~&&\label{RLY5}\mu(\Courant{x,y,z},w)+\mu(z,\Courant{x,y,w})=[D_{\rho,\mu}(x,y),\mu(z,w)],
\end{eqnarray*}
where $D_{\rho,\mu}$ is given by
\begin{eqnarray}
~ &&\label{rep} D_{\rho,\mu}(x,y)=\mu(y,x)-\mu(x,y)+[\rho(x),\rho(y)]-\rho([x,y]),\quad \forall x,y\in \g.
\end{eqnarray}
It is easy to see that $D_{\rho,\mu}$ is skew-symmetric. We denote a representation of $\g$ by $(V;\rho,\mu)$. In the sequel, we write $D_{\rho,\mu}$ as $D$ for short without confusion.
\end{defi}

Note that the notion of representations of \LYA s is also a generalization of that of Lie algebras or Lie triple systems. By a direct computation, we have the following proposition.
\emptycomment{
\begin{rmk}\label{rmk:rep}
Let $(\g,[\cdot,\cdot],\Courant{\cdot,\cdot,\cdot})$ be a Lie-Yamaguti algebra and $(V;\rho,\mu)$ its representation. If $\rho=0$ and the Lie-Yamaguti algebra $\g$ reduces to a Lie tripe system $(\g,\Courant{\cdot,\cdot,\cdot})$,  then $(V;\mu)$  is a representation of the Lie triple system $(\g,\Courant{\cdot,\cdot,\cdot})$; If $\mu=0$, $D_{\rho,\mu}=0$ and the Lie-Yamaguti algebra $\g$ reduces to a Lie algebra $(\g,[\cdot,\cdot])$, then $(V;\rho)$ is a representation  of the Lie algebra $(\g,[\cdot,\cdot])$. So the notion of representations of \LYA s is also a natural generalization of that of Lie algebras and Lie triple systems.
\end{rmk}}

\begin{pro}
If $(V;\rho,\mu)$ is a representation of a Lie-Yamaguti algebra $(\g,[\cdot,\cdot],\Courant{\cdot,\cdot,\cdot})$. Then we have the following equalities
\begin{eqnarray*}
\label{RLYc}&&D([x,y],z)+D([y,z],x)+D([z,x],y)=0;\\
\label{RLY5a}&&D(\Courant{x,y,z},w)+D(z,\Courant{x,y,w})=[D(x,y),D(z,w)];\\
~ &&\mu(\Courant{x,y,z},w)=\mu(x,w)\mu(z,y)-\mu(y,w)\mu(z,x)-\mu(z,w)D(x,y),\label{RLY6}
\end{eqnarray*}
for all $x,y,z,w\in \g$.
\end{pro}

\begin{ex}\label{ad}
Let $(\g,[\cdot,\cdot],\Courant{\cdot,\cdot,\cdot})$ be a Lie-Yamaguti algebra. We define linear maps $\ad:\g \to \gl(\g)$ and $\huaR :\otimes^2\g \to \gl(\g)$ to be $x \mapsto \ad_x$ and $(x,y) \mapsto \huaR(x,y)$ respectively, where $\ad_xz=[x,z]$ and $\huaR(x,y)z=\Courant{z,x,y}$ for all $z \in \g$. Then $(\ad,\huaR)$ forms a representation of $\g$ on itself, where $\huaL:= D_{\ad,\huaR}$ is given by
$$\huaL(x,y)z=\Courant{x,y,z}, \quad \forall z\in \g.$$
The representation $(\g;\ad,\huaR)$ is called the {\bf adjoint representation}. If $(\g^*,[\cdot,\cdot]_*,\Courant{\cdot,\cdot,\cdot}_*)$ is also a \LYA, then the adjoint representation is denoted by $(\g^*;\frkad,\frkR)$ in this paper, where \emptycomment{for any $\xi,\eta\in \g^*$, $\frkad_\xi:\g^*\longrightarrow\g^*$ and $\frkR(\xi,\eta):\g^*\longrightarrow\g^*$ is given by
$$\frkad_\xi\zeta=[\xi,\zeta]_*,\quad \frkR(\xi,\eta)\zeta=\Courant{\zeta,\xi,\eta}_*,\quad \forall\zeta\in \g^*$$
respectively. Note also that $\frkL:=D_{\frkad,\frkR}$ is given by}
$\frkL:=D_{\frkad,\frkR}.$
\end{ex}

The coadjoint representation of a \LYA ~plays an important role in the article. It is natural to recall dual representations in \cite{SZ1}.  Let $(V;\rho,\mu)$ be a representation of a Lie-Yamaguti algebra $(\g,[\cdot,\cdot],\Courant{\cdot,\cdot,\cdot})$ and $V^*$ the dual space of $V$. We define linear maps $\rho^*:\g \to \gl(V^*)$ and $\mu^*,~D_{\rho,\mu}^*:\otimes^2 \g \to \gl(V^*)$ to be
\begin{eqnarray*}
\label{pair1}\pair{\rho^*(x)\alpha, v}&=&-\pair{\alpha,\rho(x)v},\\
\label{pair2}\pair{\mu^*(x,y)\alpha,v}&=&-\pair{\alpha,\mu(x,y)v},\\
\label{pair3}\langle D_{\rho,\mu}^*(x,y)\alpha,v\rangle&=&-\langle\alpha,D_{\rho,\mu}(x,y)v\rangle.
\end{eqnarray*}
for all $x,y \in \g, ~\alpha \in V^*, ~v \in V$.

\begin{pro}\label{dual}{\rm (\cite{SZ1})}
Let $(V;\rho,\mu)$ be a representation of a Lie-Yamaguti algebra $(\g,[\cdot,\cdot],\Courant{\cdot,\cdot,\cdot})$. Then
\begin{eqnarray*}
\big(V^*;\rho^*,-\mu^*\tau\big)
\end{eqnarray*}
is a representation of   $\g$ on $V^*$, where $D_{\rho,\mu}^*=D_{\rho^*,-\mu^*\tau}$. We call $(V^*;\rho^*,-\mu^*\tau)$ the {\bf dual representation} of $(V;\rho,\mu)$.
\end{pro}

The coadjoint representation of a \LYA ~is dual to the adjoint representation.
\begin{ex}
 Let $(\g,[\cdot,\cdot],\Courant{\cdot,\cdot,\cdot})$ be a Lie-Yamaguti algebra and $(\g;\ad,\huaR)$ its adjoint representation, where $\ad, \huaR$ are given in Example \ref{ad}. Then $(\g^*;\ad^*,-\huaR^*{\tau})$ is the dual representation of the adjoint representation, called the {\bf coadjoint representation}. Note that $\huaL^*:=D_{\ad^*,-\huaR^*\tau}$ is dual to $-\huaL$, i.e.,
 $$\pair{\huaL^*(x,y)\alpha,z}=-\pair{\alpha,\Courant{x,y,z}},\quad\forall x,y,z\in \g,~\alpha\in \g^*.$$
 If $(\g^*,[\cdot,\cdot]_*,\Courant{\cdot,\cdot,\cdot}_*)$ is a \LYA, and $(\g^*,\frkad,\frkR)$ is its adjoint representation, then the coadjoint representation of $(\g^*,\frkad,\frkR)$ is $(\g;\frkad^*,-\frkR^*\tau)$, where $\frkL^*=D_{\frkad^*,-\frkR^*\tau}$.
\end{ex}

Representations of a Lie-Yamaguti algebra can be characterized by the semidirect product Lie-Yamaguti algebras.

\begin{pro}\cite{Zhang1}
Let $(\g,[\cdot,\cdot],\Courant{\cdot,\cdot,\cdot})$ be a Lie-Yamaguti algebra and $V$ a vector space. Suppose that $\rho:\g \to \gl(V)$ and $\mu:\otimes^2 \g \to \gl(V)$ are linear maps. Then $(V;\rho,\mu)$ is a representation of $(\g,[\cdot,\cdot],\Courant{\cdot,\cdot,\cdot})$ if and only if there is a Lie-Yamaguti algebra structure $([\cdot,\cdot]_{\ltimes},\Courant{\cdot,\cdot,\cdot}_{\ltimes})$ on the direct sum $\g \oplus V$ which is defined to be
\begin{eqnarray*}
\label{semi1}[x+u,y+v]_{\ltimes}&=&[x,y]+\rho(x)v-\rho(y)u,\\
\label{semi2}~\Courant{x+u,y+v,z+w}_{\ltimes}&=&\Courant{x,y,z}+D(x,y)w+\mu(y,z)u-\mu(x,z)v,
\end{eqnarray*}
for all $x,y,z \in \g, ~u,v,w \in V$. This Lie-Yamaguti algebra $(\g \oplus V,[\cdot,\cdot]_{\ltimes},\Courant{\cdot,\cdot,\cdot}_{\ltimes})$ is called the {\bf semidirect product Lie-Yamaguti algebra}, and is denoted by $\g \ltimes_{\rho,\mu} V$.
\end{pro}

The cohomology theory of \LYA s ~was founded in \cite{Yamaguti2}. Let $(V;\rho,\mu)$ be a representation of a Lie-Yamaguti algebra $(\g,[\cdot,\cdot],\Courant{\cdot,\cdot,\cdot})$.
\begin{itemize}
\item the set of $p$-cochains is denoted by $C^p_{\rm LieY}(\g,V)~(p \geqslant 1)$, where
\begin{eqnarray*}
C^{n+1}_{\rm LieY}(\g,V):=
\begin{cases}
\Hom(\underbrace{\wedge^2\g\otimes \cdots \otimes \wedge^2\g}_n,V)\times \Hom(\underbrace{\wedge^2\g\otimes\cdots\otimes\wedge^2\g}_{n}\otimes\g,V), & \forall n\geqslant 1,\\
\Hom(\g,V), &n=0.
\end{cases}
\end{eqnarray*}
\item the coboundary map of $p$-cochains $\de:C^{n+1}_{\rm LieY}(\g,V)\longrightarrow C^{n+2}_{\rm LieY}(\g,V) ~(n\geqslant 0)$ is defined to be
\begin{itemize}
\item[(1)] If $n\geqslant 1$, for any $(f,g)\in C^{n+1}_{\rm LieY}(\g,V)$, the coboundary map
$$\de=(\de_{\rm I},\de_{\rm II}):C^{n+1}_{\rm LieY}(\g,V)\to C^{n+2}_{\rm LieY}(\g,V),$$
$$\qquad \qquad\qquad \qquad\qquad \quad (f,g)\mapsto\Big(\de_{\rm I}(f,g),\de_{\rm II}(f,g)\Big),$$
 is given as follows
\begin{eqnarray*}
~\nonumber &&\Big(\de_{\rm I}(f,g)\Big)(\frkX_1,\cdots,\frkX_{n+1})\\
~\label{cohomo1} &=&(-1)^n\Big(\rho(x_{n+1})g(\frkX_1,\cdots,\frkX_n,y_{n+1})-\rho(y_{n+1})g(\frkX_1,\cdots,\frkX_n,x_{n+1})\\
~\nonumber &&-g(\frkX_1,\cdots,\frkX_n,[x_{n+1},y_{n+1}])\Big)\\
~\nonumber &&+\sum_{k=1}^{n}(-1)^{k+1}D_{\rho,\mu}(\frkX_k)f(\frkX_1,\cdots,\hat{\frkX_k},\cdots,\frkX_{n+1})\\
~\nonumber &&+\sum_{1\leqslant k<l\leqslant n+1}(-1)^{k}f(\frkX_1,\cdots,\hat{\frkX_k},\cdots,\frkX_k\circ\frkX_l,\cdots,\frkX_{n+1}),\\
~ \nonumber&&\\
~\nonumber &&\Big(\de_{\rm II}(f,g)\Big)(\frkX_1,\cdots,\frkX_{n+1},z)\\
~\label{cohomo2}&=&(-1)^n\Big(\mu(y_{n+1},z)g(\frkX_1,\cdots,\frkX_n,x_{n+1})-\mu(x_{n+1},z)g(\frkX_1,\cdots,\frkX_n,y_{n+1})\Big)\\
~\nonumber &&+\sum_{k=1}^{n+1}(-1)^{k+1}D_{\rho,\mu}(\frkX_k)g(\frkX_1,\cdots,\hat{\frkX_k},\cdots,\frkX_{n+1},z)\\
~\nonumber &&+\sum_{1\leqslant k<l\leqslant n+1}(-1)^kg(\frkX_1,\cdots,\hat{\frkX_k},\cdots,\frkX_k\circ\frkX_l,\cdots,\frkX_{n+1},z)\\
~\nonumber &&+\sum_{k=1}^{n+1}(-1)^kg(\frkX_1,\cdots,\hat{\frkX_k},\cdots,\frkX_{n+1},\Courant{x_k,y_k,z}),
\end{eqnarray*}
where $\frkX_i=x_i\wedge y_i\in\wedge^2\g~(i=1,\cdots,n+1),~z\in \g$, and the notation $\frkX_k\circ\frkX_l$ means that
$$\frkX_k\circ\frkX_l:=\Courant{x_k,y_k,x_l}\wedge y_l+x_l\wedge\Courant{x_k,y_k,y_l}.$$

\item[(2)] If $n=0$,  for any element $f \in C^1_{\rm LieY}(\g,V)$, the coboundary map
$$\de:C^1_{\rm LieY}(\g,V)\to C^2_{\rm LieY}(\g,V),$$
$$\qquad \qquad \qquad f\mapsto \Big(\de_I(f),\de_{II}(f)\Big),$$
is given by
\begin{eqnarray*}
\label{1cochain}\Big(\de_{\rm I}(f)\Big)(x,y)&=&\rho(x)f(y)-\rho(y)f(x)-f([x,y]),\\
~ \label{2cochain}\Big(\de_{\rm II}(f)\Big)(x,y,z)&=&D_{\rho,\mu}(x,y)f(z)+\mu(y,z)f(x)-\mu(x,z)f(y)-f(\Courant{x,y,z}),\quad \forall x,y, z\in \g.
\end{eqnarray*}
\end{itemize}
\end{itemize}

In particular, we obtain the precise formula of $1$-cocycle.
\begin{defi}
Let $(\g,\br,,\ltp)$ be a \LYA ~and $(V;\rho,\mu)$ a representation of $\g$. A linear map $f:\g\longrightarrow V$ is called a {\bf $1$-cocycle} of $\g$ with respect to $(V;\rho,\mu)$ if $f$ satisfies
\begin{eqnarray*}
f([x,y])&=&\rho(x)f(y)-\rho(y)f(x),\\
f(\Courant{x,y,z})&=&D(x,y)f(z)+\mu(y,z)f(x)-\mu(x,z)f(y),\quad\forall x,y,z\in \g.
\end{eqnarray*}
\end{defi}

\begin{ex}
A {\bf derivation} on a \LYA ~$(\g,\br,,\ltp)$ is a linear map $\Delta:\g\longrightarrow\g$ such that
 \begin{eqnarray*}
 \Delta([x,y])&=&[\Delta(x),y]+[x,\Delta(y)],\\
 \Delta(\Courant{x,y,z})&=&\Courant{\Delta(x),y,z}+\Courant{x,\Delta(y),z}+\Courant{x,y,\Delta(z)},\quad\forall x,y,z \in \g.
 \end{eqnarray*}
Thus a derivation is a $1$-cocycle of $\g$ with respect to the adjoint representation $(\g;\ad,\huaR)$.
\end{ex}

\smallskip
\section{Relative Rota-Baxter operators, the classical Yang-Baxter equation, and local cocycle Lie-Yamaguti bialgebras}
In this section, we define the classical Yang-Baxter equation in \LYA s and clarify the relationship between its solutions and relative Rota-Baxter operators. Moreover as byproducts, we generalize conclusions given by Bai and Semonov-Tian-Shansky. Finally, we give the definition of local cocycle Lie-Yamaguti bialgebras. 
First of all, let us recall some notions and conclusions in \cite{SZ1} of relative Rota-Baxter operators and pre-\LYA s.
\begin{defi}{\rm (\cite{SZ1})}
Let $(\g,[\cdot,\cdot],\Courant{\cdot,\cdot,\cdot})$ be a Lie-Yamaguti algebra with a representation $(V;\rho,\mu)$ and $T:V\to \g$ a linear map. If $T$ satisfies
\begin{eqnarray*}
[Tu,Tv]&=&T\Big(\rho(Tu)v-\rho(Tv)u\Big),\\
\Courant{Tu,Tv,Tw}&=&T\Big(D(Tu,Tv)w+\mu(Tv,Tw)u-\mu(Tu,Tw)v\Big), \quad\forall u,v,w \in V,
\end{eqnarray*}
then we call $T$ a {\bf relative Rota-Baxter operator} on $\g$ with respect to the representation $(V;\rho,\mu)$.
\end{defi}

\begin{defi}{\rm (\cite{SZ1})}
A {\bf pre-Lie-Yamaguti algebra} is a vector space $A$ with a bilinear operation $*:\otimes^2A \to A$ and a trilinear operation $\{\cdot,\cdot,\cdot\} :\otimes^3A \to A$ such that for all $x,y,z,w,t \in A$
\begin{eqnarray}
~ &&\label{pre2}\{z,[x,y]_C,w\}-\{y*z,x,w\}+\{x*z,y,w\}=0,\\
~ &&\label{pre4}\{x,y,[z,w]_C\}=z*\{x,y,w\}-w*\{x,y,z\},\\
~ &&\label{pre5}\{\{x,y,z\},w,t\}-\{\{x,y,w\},z,t\}-\{x,y,\{z,w,t\}_D\}-\{x,y,\{z,w,t\}\}\\
~ &&\nonumber+\{x,y,\{w,z,t\}\}+\{z,w,\{x,y,t\}\}_D=0,\\
~ &&\label{pre6}\{z,\{x,y,w\}_D,t\}+\{z,\{x,y,w\},t\}-\{z,\{y,x,w\},t\}+\{z,w,\{x,y,t\}_D\}\\
~ &&\nonumber+\{z,w,\{x,y,t\}\}-\{z,w,\{y,x,t\}\}=\{x,y,\{z,w,t\}\}_D-\{\{x,y,z\}_D,w,t\},\\
~&&\label{pre7}\{x,y,z\}_D*w+\{x,y,z\}*w-\{y,x,z\}*w=\{x,y,z*w\}_D-z*\{x,y,w\}_D,
\end{eqnarray}
where the commutator
$[\cdot,\cdot]_C:\wedge^2\g \to \g$ and $\{\cdot,\cdot,\cdot\}_D: \otimes^3 A \to A$ are defined by for all $x,y,z \in A$,
\begin{eqnarray}
~[x,y]_C:=x*y-y*x, \quad \forall x,y \in A,\label{pre10}
\end{eqnarray}
and
\begin{eqnarray}
\{x,y,z\}_D:=\{z,y,x\}-\{z,x,y\}+(y,x,z)-(x,y,z), \label{pre3}
\end{eqnarray}
respectively. Here $(\cdot,\cdot,\cdot)$ denotes the associator: $(x,y,z):=(x*y)*z-x*(y*z)$. It is obvious that $\{\cdot,\cdot,\cdot\}_D$ is skew-symmetric with respect to the first two variables. We denote a pre-Lie-Yamaguti algebra by $(A,*,\{\cdot,\cdot,\cdot\})$.
\end{defi}

Let $(A,*,\{\cdot,\cdot,\cdot\})$ be a pre-Lie-Yamaguti algebra. Define
\begin{itemize}
\item a pair of operations $(\br_{_C},\ltp_C)$ to be
\begin{eqnarray*}
\label{subLY1}[x,y]_C&=&x*y-y*x,\\
~\label{subLY2}\Courant{x,y,z}_C&=&\{x,y,z\}_D+\{x,y,z\} -\{y,x,z\}, \quad \forall x,y,z \in \g,
\end{eqnarray*}
where $\{\cdot,\cdot,\cdot\}_D$ is given by \eqref{pre3}.
\item linear maps
$$\Ad: A \to \gl(A), \quad R: \otimes^2A \to \gl(A)$$
 to be
 $$x \mapsto \Ad_x, \quad (x,y)\mapsto R(x,y)$$
 respectively, where $\Ad_xz=x*z$ and $R(x,y)z=\{z,x,y\}$ for all $z \in A$.
\end{itemize}

The following proposition is the Theorem 3.11 in \cite{SZ1}.

\begin{pro}\label{subad}{\rm (\cite{SZ1})}
With the above notations, then  we have
\begin{itemize}
\item [\rm (i)] the operation $(\br_{_C},\ltp_C)$ defines a Lie-Yamaguti algebra structure on $A$. This Lie-Yamaguti algebra $(A,\br_{_C},\ltp_C)$ is called the {\bf sub-adjacent \LYA} and is denoted by $A^c$;
\item [\rm (ii)] the triple $(A;\Ad,R)$ is a representation of the sub-adjacent Lie-Yamaguti algebra $A^c$ on $A$. Furthermore, the identity map $\Id:A\longrightarrow A$ is a relative Rota-Baxter operator on $A^c$ with respect to the representation $(A;\Ad,R)$, where
    $$L:=D_{\Ad,R}:\wedge^2A\longrightarrow \gl(A),\quad (x,y)\mapsto L(x,y)$$
    is given by
    $$L(x,y)z=\{x,y,z\}_D,\quad\forall z\in A.$$
\end{itemize}
\end{pro}

Next, we introduce some notations and terminologies. In this section, by $r=\sum_ix_i\otimes y_i\in \otimes^2\g$ we always mean a $2$-tensor.
First, $r=\sum_ix_i\otimes y_i\in \otimes^2\g$ can be embedded into an $n$-tensor $r_{pq}\in \otimes^n\g~(n\geqslant 2)$ in the following rule:
$$r_{pq}:=\sum_i z_{i1}\otimes\cdots\otimes z_{in},$$
where
\begin{eqnarray*}
z_{ij}=
\begin{cases}
x_i, & j=p,\\
y_i, & j=q,\\
1, &  i\neq p,q,
\end{cases}
\end{eqnarray*}
for any $1\leqslant p\neq q\leqslant n$.

Let $(\g,\br,,\ltp)$ be a \LYA, we define $[r,r]\in \otimes^3\g$ and $\Courant{r,r,r}\in \otimes^4\g$ respectively to be
\begin{eqnarray}
\label{r1}[r,r]&=&[r_{12},r_{13}]+[r_{12},r_{23}]+[r_{13},r_{23}],\\
~\label{r2}\Courant{r,r,r}&=&\Courant{r_{31},r_{32},r_{43}}+\Courant{r_{13},r_{41},r_{12}}+\Courant{r_{42},r_{23},r_{21}}+\Courant{r_{41},r_{42},r_{43}}.
\end{eqnarray}
Here ${r_{pq}}'s$ in Eqs. \eqref{r1} and \eqref{r2} are the embedded $3$-tensor and the $4$-tensor by $r\in \otimes^2\g$ respectively.
More precisely, we have
\begin{eqnarray*}
[r,r]&=&\sum_{ij}\Big([x_i,x_j]\otimes y_i\otimes y_j+x_i\otimes[y_i,x_j]\otimes y_j+x_i\otimes x_j\otimes[y_i,y_j]\Big),\\
\Courant{r,r,r}&=&\sum_{ijk}\Big(\llbracket y_k,x_i,x_j\rrbracket \otimes y_i\otimes y_j\otimes x_k+y_j\otimes\llbracket y_k,x_i.x_j\rrbracket\otimes y_i\otimes x_k\\
~ &&~~+y_i\otimes y_j\otimes\llbracket x_i,x_j,y_k\rrbracket\otimes x_k+y_i\otimes y_j\otimes y_k\otimes\llbracket x_i,x_j,x_k\rrbracket\Big).
\end{eqnarray*}

Define two linear maps $\delta:\g\longrightarrow\otimes^2\g$ and $\omega:\g\longrightarrow\otimes^3\g$ respectively to be
\begin{eqnarray}
\label{cocycle1}\delta(x)&:=&\sum_i\Big([x,x_i]\otimes y_i+x_i\otimes[x,y_i]\Big),\\
\label{cocycle2}\omega(x)&:=&\sum_{ij}\Big(\llbracket x,x_i,x_j \rrbracket\otimes y_j\otimes y_i+y_j\otimes\llbracket x_i,x,x_j \rrbracket\otimes y_i
+y_j\otimes y_i\otimes \llbracket x_i,x_j,x \rrbracket\Big), ~~\forall x\in \g.
\end{eqnarray}
In the sequel, two linear operations $\delta^*:\otimes^2\g^*\longrightarrow\g^*$ and $\omega^*:\otimes^3\g^*\longrightarrow\g^*$ are denoted by $\br_{_*}$ and $\ltp_*$ respectively.

Set
\begin{eqnarray}\label{c1}
\begin{cases}
\delta_1(x)&:=\sum_ix_i\otimes[x,y_i],\\
\delta_2(x)&:=\sum_i[x,x_i]\otimes y_i,
\end{cases}
\end{eqnarray}
and
\begin{eqnarray}\label{c2}
\begin{cases}
\omega_1(x)&:=\sum_{ij}\llbracket x,x_i,x_j \rrbracket\otimes y_j\otimes y_i,\\
\omega_2(x)&:=\sum_{ij}y_j\otimes\llbracket x_i,x,x_j \rrbracket\otimes y_i,\\
\omega_3(x)&:=\sum_{ij}y_j\otimes y_i\otimes \llbracket x_i,x_j,x \rrbracket,
\end{cases}
\end{eqnarray}
for all $x\in \g$.

\begin{pro}\label{skew}
Let $(\g,\br,,\ltp)$ be a \LYA ~and $r\in \otimes^2\g$. Suppose that $r$ is skew-symmetric, and that $\delta$ and $\omega$ are induced by $r$ as in Eqs. \eqref{cocycle1} and \eqref{cocycle2}. Then $\delta^*:\otimes^2\g^*\longrightarrow\g^*$ is skew-symmetric and $\omega^*:\otimes^3\g^*\longrightarrow\g^*$ is skew-symmetric in the first two variables.
\end{pro}
\begin{proof}
Indeed, for any $x\in \g$, we have
\begin{eqnarray*}
\sigma_{12}\omega_1(x)&=&\sum_{ij}y_j\otimes \llbracket x,x_i,x_j \rrbracket\otimes y_i=-\sum_{ij}y_j\otimes\llbracket x_i,x,x_j \rrbracket\otimes y_i=-\omega_2(x),\\
\sigma_{12}\omega_2(x)&=&\sum_{ij}\llbracket x_i,x,x_j \rrbracket\otimes y_j\otimes y_i=-\sum_{ij}\llbracket x,x_i,x_j \rrbracket\otimes y_j\otimes y_i=-\omega_1(x),\\
\sigma_{12}\omega_3(x)&=&\sum_{ij}y_i\otimes y_j\otimes\llbracket x_i,x_j,x \rrbracket=-\sum_{ij}y_i\otimes y_j\otimes\llbracket x_j,x_i,x \rrbracket=-\omega_3(x).
\end{eqnarray*}
This shows that $\omega^*$ is skew-symmetric in the first two variables. Moreover, since $r$ is skew-symmetric, we have $\sigma_{12}\delta_1(x)=-\delta_2(x)$ and  $\sigma_{12}\delta_2(x)=-\delta_1(x)$ for any $x\in \g$, and thus $\delta^*$ is skew-symmetric. This finishes the proof.
\end{proof}

A $2$-tensor $r$ induces a linear map $r^\sharp:\g^*\longrightarrow\g$ defined to be
\begin{eqnarray}
\langle r^\sharp(\xi),\eta\rangle=\langle r,\xi\otimes\eta\rangle,\quad \forall \xi,\eta\in \g^*.\label{rsharp}
\end{eqnarray}
Similarly, a $2$-tensor $\huaB\in\otimes^2\g^*$ induces a linear map $\huaB^\natural:\g\longrightarrow\g^*$ defined by
\begin{eqnarray}
\langle\huaB^\natural(x),y\rangle=\huaB(x,y), \quad \forall x,y\in \g.\label{Bna}
\end{eqnarray}

\begin{pro}\label{pro:fund}
Let $(\g,\br,,\ltp)$ be a \LYA~ and $r=\sum_ix_i\otimes y_i$. Suppose that $r$ is skew-symmetric, and that $\delta:\g\longrightarrow\otimes^2\g$ and $\omega:\g\longrightarrow\otimes^3\g$ are defined by $r$ as in Eqs. \eqref{cocycle1} and \eqref{cocycle2} respectively. Then we have
\begin{eqnarray*}
[\xi,\eta]_*&=&\ad^*_{r^\sharp(\xi)}\eta-\ad^*_{r^\sharp(\eta)}\xi,\\
\Courant{\xi,\eta,\zeta}_*&=&\huaL^*(r^\sharp(\xi),r^\sharp(\eta))\zeta-\huaR^*(r^\sharp(\zeta),r^\sharp(\eta))\xi+\huaR^*(r^\sharp(\zeta),r^\sharp(\xi))\eta,
\quad \forall \xi,\eta,\zeta\in \g^*.
\end{eqnarray*}
\end{pro}
\begin{proof}
It is sufficient to prove that
\begin{eqnarray}
\label{br1}\langle\delta(x),\xi\otimes\eta\rangle&=&\langle x,[\xi,\eta]_*\rangle,\\
\label{br2}\langle\omega(x),\xi\otimes\eta\otimes\zeta\rangle&=&\langle x,\Courant{\xi,\eta,\zeta}_*\rangle,\quad \forall x\in \g^*,\xi,\eta,\zeta\in \g^*.
\end{eqnarray}
Let $r=\sum_ix_i\otimes y_i$. Since $r$ is skew-symmetric, we have
\begin{eqnarray*}
\langle x,\huaR^*(r^\sharp(\zeta),r^\sharp(\eta))\xi\rangle&=&-\langle\Courant{x,r^\sharp(\zeta),r^\sharp(\eta)},\xi\rangle=-\langle r^\sharp(\zeta),\huaR^*(x,r^\sharp(\eta))\xi\rangle\\
~ &=&-\langle r,\zeta\otimes\huaR^*(x,r^\sharp(\eta))\xi\rangle=-\sum_i\langle y_i,\zeta\rangle\langle x_i,\huaR^*(x,r^\sharp(\eta))\xi\rangle\\
~ &=&-\sum_i\langle y_i,\zeta\rangle\langle r^\sharp(\eta),\huaL^*(x_i,x)\xi\rangle=-\sum_i\langle y_i,\zeta\rangle\langle r,\eta\otimes \huaL^*(x_i,x)\xi\rangle\\
~ &=&\sum_{ij}\langle y_i,\zeta\rangle\langle y_j,\eta\rangle\langle x_j,\huaL^*(x_i,x)\xi\rangle\\
~ &=&-\langle\sum_{ij}\Courant{x,x_i,x_j}\otimes y_j\otimes y_i,\xi\otimes\eta\otimes\zeta\rangle\\
~ &=&-\langle\omega_1(x),\xi\otimes\eta\otimes\zeta\rangle.
\end{eqnarray*}
Hence, we obtain that
$$-\langle x,\huaR^*(r^\sharp(\zeta),r^\sharp(\eta))\xi\rangle=\langle\omega_1(x),\xi\otimes\eta\otimes\zeta\rangle.$$
Moreover, we also have that
\begin{eqnarray*}
\langle x,\huaL^*(r^\sharp(\xi),r^\sharp(\eta))\zeta\rangle&=&-\pair{\Courant{r^\sharp(\xi),r^\sharp(\eta),x},\zeta}\\
~ &=&\pair{r^\sharp(\xi),\huaR^*(r^\sharp(\eta),x)\zeta}\\
~ &=&\pair{r,\xi\otimes\huaR^*(r^\sharp(\eta),x)\zeta}\\
~ &=&\sum_j\pair{y_j,\xi}\pair{x_j,\huaR^*(r^\sharp(\eta),x)\zeta}\\
~ &=&-\sum_j\pair{y_j,\xi}\pair{r^\sharp(\eta),\huaR^*(x_j,x)\zeta}\\
~ &=&-\sum_j\pair{y_j,\xi}\pair{r,\eta\otimes\huaR^*(x_j,x)\zeta}\\
~ &=&-\sum_{ij}\pair{y_j,\xi}\pair{y_i,\eta}\pair{x_i,\huaR^*(x_j,x)\zeta}\\
~ &=&\pair{\sum_{ij}y_j\otimes y_i\otimes\Courant{x_i,x_j,x},\xi\otimes\eta\otimes\zeta}\\
~ &=&\pair{\omega_3(x),\xi\otimes\eta\otimes\zeta},\\
~ \pair{x,\huaR^*(r^\sharp(\zeta),r^\sharp(\xi))\eta}&=&-\pair{\Courant{x,r^\sharp(\zeta),r^\sharp(\xi)},\eta}\\
~ &=&-\pair{r^\sharp(\zeta),\huaR^*(x,r^\sharp(\xi))\eta}\\
~ &=&-\pair{r,\zeta\otimes\huaR^*(x,r^\sharp(\xi))\eta}\\
~ &=&-\sum_i\pair{y_i,\zeta}\pair{x_i,\huaR^*(x,r^\sharp(\xi))\eta}\\
~ &=&-\sum_i\pair{y_i,\zeta}\pair{r^\sharp(\xi),\huaL^*(x_i,x)\eta}\\
~ &=&-\sum_i\pair{y_i,\zeta}\pair{r,\xi\otimes\huaL^*(x_i,x)\eta}\\
~ &=&\sum_{ij}\pair{y_i,\zeta}\pair{y_j,\xi}\pair{x_j,\huaL^*(x_i,x)\eta}\\
~ &=&\pair{\sum_{ij}y_j\otimes\Courant{x_i,x,x_j}\otimes y_i,\xi\otimes\eta\otimes\zeta}\\
~ &=&\pair{\omega_2(x),\xi\otimes\eta\otimes\zeta}.
\end{eqnarray*}
This gives Eq. \eqref{br2}. And Eq. \eqref{br1} can be proved similarly, so we omit the details. This finishes the proof.
\end{proof}

\begin{thm}\label{THM:CYBE}
Let $(\g,\br,,\ltp)$ be a \LYA ~and $r\in \otimes^2\g$ skew-symmetric and nondegenerate. Then $r$ satisfies
\begin{eqnarray}\label{CYBE}
\begin{cases}
~[r,r]&=0,\\
~\Courant{r,r,r}&=0,
\end{cases}
\end{eqnarray}
if and only if $r^\sharp:\g^*\longrightarrow\g$ is a relative Rota-Baxter operator on $(\g,\br,,\ltp)$ with respect to the coadjoint representation $(\g^*;\ad^*,-\huaR^*\tau)$, where $[r,r]$ and $\Courant{r,r,r}$ are defined as in Eqs. \eqref{r1} and \eqref{r2} respectively.
\end{thm}
\begin{proof}
Let $\xi,\eta,\zeta\in \g^*$ and $r=\sum_ix_i\otimes y_i$. Then we compute that
\begin{eqnarray*}
\langle\xi\otimes\eta,r\rangle&=&\sum_i\langle\xi,x_i\rangle\langle\eta,y_i\rangle=\langle\xi,\sum_i\langle\eta,y_i\rangle x_i\rangle.
\end{eqnarray*}
The skew-symmetry of $r$ yields that
$$T(\xi)=\sum_i\langle\xi,y_i\rangle x_i.$$
Now we compute that
\begin{eqnarray*}
T\Big(\huaL^*(T(\xi),T(\eta))\zeta\Big)&=&T\Big(\huaL^*\big(\sum_i\langle\xi,y_i\rangle x_i,\sum_j\langle\eta,y_j\rangle x_j\big)\zeta\Big)\\
~ &=&\sum_{ij}\langle\xi,y_i\rangle\langle\eta,y_j\rangle T\Big(\frkL^*(x_i,x_j)\zeta\Big)\\
~ &=&\sum_{ij}\langle\xi,y_i\rangle\langle\eta,y_j\rangle \sum_k\langle\huaL^*(x_i,x_j)\zeta,y_k\rangle x_k\\
~ &=&-\sum_{ijk}\langle\xi,y_i\rangle\langle\eta,y_j\rangle \langle\zeta,\Courant{x_i,x_j,y_k}\rangle x_k\\
~ &=&-\Big(\langle\xi,\cdot\rangle\otimes\langle\eta,\cdot\rangle\otimes\langle\zeta,\cdot\rangle\otimes1\Big)\Courant{r_{31},r_{32},r_{43}},
\end{eqnarray*}
Similarly, we also have that
\begin{eqnarray*}
-T\Big(\huaR^*(T(\zeta),T(\eta))\xi\Big)&=&-\sum_{ij}\langle\zeta,y_i\rangle\langle\eta,y_j\rangle T\Big(\huaR^*(x_i,x_j)\xi\Big)\\
~ &=&\sum_{ijk}\langle\zeta,y_i\rangle\langle\eta,y_j\rangle \langle\xi,\Courant{y_k,x_i,x_j}\rangle x_k\\
~ &=&-\Big(\langle\xi,\cdot\rangle\otimes\langle\eta,\cdot\rangle\otimes\langle\zeta,\cdot\rangle\otimes1\Big)\Courant{r_{13},r_{41},r_{12}},\\
T\Big(\huaR^*(T(\zeta),T(\xi))\eta\Big)&=&\sum_{ij}\langle\zeta,y_i\rangle\langle\xi,y_j\rangle T\Big(\huaR^*(x_i,x_j)\eta\Big)\\
~ &=&\sum_{ijk}\langle\zeta,y_i\rangle\langle\xi,y_j\rangle \langle\eta,\Courant{y_k,x_i,x_j}\rangle x_k\\
~ &=&-\Big(\langle\xi,\cdot\rangle\otimes\langle\eta,\cdot\rangle\otimes\langle\zeta,\cdot\rangle\otimes1\Big)\Courant{r_{42},r_{23},r_{21}},\\
\Courant{T(\xi),T(\eta),T(\zeta)}&=&\sum_{ijk}\Courant{\langle\xi,y_i\rangle x_i,\langle\eta,y_j\rangle x_j,\langle\zeta,y_k\rangle x_k}\\
~&=&\sum_{ijk}\langle\xi,y_i\rangle \langle\eta,y_j\rangle\langle\zeta,y_k\rangle\Courant{x_i,x_j,x_k}\\
~ &=&\Big(\langle\xi,\cdot\rangle\otimes\langle\eta,\cdot\rangle\otimes\langle\zeta,\cdot\rangle\otimes1\Big)
\Courant{r_{41},r_{42},r_{43}}.
\end{eqnarray*}
Thus we obtain that
$$\pair{\kappa,\Courant{T(\xi),T(\eta),T(\zeta)}-T\Big(\huaL^*(T(\xi),T(\eta))\zeta-\huaR^*(T(\zeta),T(\eta))\xi+\huaR^*(T(\zeta),T(\xi))\eta\Big)}
=\pair{\xi\otimes\eta\otimes\zeta\otimes\kappa,\Courant{r,r,r}}.$$
Similarly, we also have the following relation
$$\pair{\kappa,[T(\xi),T(\eta)]-T\Big(\ad^*_{T(\xi)}\eta-\ad^*_{T(\eta)}\xi\Big)}=\pair{\xi\otimes\eta\otimes\kappa,[r,r]}.$$
The conclusion thus follows.
\end{proof}

This leads to the following definitions of the classical Yang-Baxter equation in \LYA s and the classical Lie-Yamaguti $r$-matrix.
\begin{defi}\label{rmatrix}
Let $(\g,\br,,\ltp)$ be a \LYA ~and $r\in \otimes^2\g$. The equation \eqref{CYBE} given in Theorem \ref{THM:CYBE} is called the {\bf classical Lie-Yamaguti Yang-Baxter equation} in $\g$ and $r$ is called the {\bf classical Lie-Yamaguti $r$-matrix} of $\g$.
\end{defi}

We obtain the following corollary as a direct consequence.
\begin{cor}
If $r\in \otimes^2\g$ is a skew-symmetric classical Lie-Yamaguti $r$-matrix, then the induced map $r^\sharp:\g^*\longrightarrow\g$ defined by \eqref{rsharp} is a Lie-Yamaguti homomorphism from $(\g^*,\br_{_*},\ltp_*)$ to $(\g,\br,,\ltp)$.
\end{cor}

\begin{ex}\label{ex:2di}
Let $\g$ be a $2$-dimensional \LYA ~with a basis $\{e_1,e_2\}$ defined to be
$$[e_1,e_2]=e_1,\quad\Courant{e_1,e_2,e_2}=e_1.$$
Then any skew-symmetric $2$-tensor $r=k(e_1\otimes e_2-e_2\otimes e_1)$ is a solution to the classical Lie-Yamaguti Yang-Baxter equation.
\end{ex}

We give the following interpretation of the invertible skew-symmetric classical Lie-Yamaguti $r$-matrices, which is parallel to the result for the classical Yang-Baxter equation in a Lie algebra or a $3$-Lie algebra.

\begin{pro}\label{symplectic}
Let $(\g,\br,,\ltp)$ be a \LYA ~and $r\in \otimes^2\g$. Suppose that $r$ is skew-symmetric and nondegenerate. Then $r$ is a classical $r$-matrix of $\g$ if and only if the nondegenerate, skew-symmetric bilinear form $\omega\in \wedge^2\g^*$ defined to be
$$\omega(x,y):=\langle (r^\sharp)^{-1}(x),y\rangle,\quad \forall x,y\in \g$$
is a symplectic structure\footnote{The notion of symplectic structures was introduced in \cite{SZ1}}, i.e., $\omega$ satisfies
\begin{eqnarray*}
 \omega(x,[y,z])+\omega(y,[z,x])+\omega(z,[x,y])&=&0,\\
 ~\omega(z,\Courant{x,y,w})-\omega(x,\Courant{w,z,y})+\omega(y,\Courant{w,z,x})-\omega(w,\Courant{x,y,z})&=&0,
 \end{eqnarray*}
for all $x,y,z,w\in \g$.
\end{pro}
\begin{proof}
Since $r\in \otimes^2\g$ is nondegenerate, for all $\xi,\eta,\zeta\in \g^*$, there exists $x,y,z\in \g$, such that $r^\sharp(\xi)=x,r^\sharp(\eta)=y,r^\sharp(\zeta)=z$. Then it follows from Theorem \ref{THM:CYBE} that
\begin{eqnarray*}
\omega(w,\Courant{x,y,z})&=&-\langle(r^\sharp)^{-1}(\llbracket r^\sharp(\xi),r^\sharp(\eta),r^\sharp(\zeta)\rrbracket,w)\rangle\\
~ &=&-\langle\huaL^*(r^\sharp(\xi),r^\sharp(\eta))\zeta-\huaR^*(r^\sharp(\zeta),r^\sharp(\eta))\xi+\huaR^*(r^\sharp(\zeta),r^\sharp(\xi))\eta,w\rangle\\
~ &=&\langle\zeta,\llbracket r^\sharp(\xi),r^\sharp(\eta),w\rrbracket\rangle-\langle\xi,\llbracket w,r^\sharp(\zeta),r^\sharp(\eta)\rrbracket\rangle
+\langle\eta,\llbracket w,r^\sharp(\zeta),r^\sharp(\xi)\rrbracket\rangle\\
~ &=&\omega(z,\Courant{x,y,w})-\omega(x,\Courant{w,z,y})+\omega(y,\Courant{w,z,x})
\end{eqnarray*}
and
\begin{eqnarray*}
\omega(z,[x,y])&=&-\langle(r^\sharp)^{-1}([r^\sharp(\xi),r^\sharp(\eta)]),z\rangle\\
~ &=&-\langle\ad^*_{r^\sharp(\xi)}\eta-\ad^*_{r^\sharp(\eta)}\xi,z\rangle\\
~ &=&\langle\eta,[r^\sharp(\xi),z]\rangle-\langle\xi,[r^\sharp(\eta),z]\rangle\\
~ &=&\omega(y,[x,z])-\omega(x,[y,z]).
\end{eqnarray*}
This finishes the proof.
\end{proof}

Given a $2$-tensor $\overline T\in V^*\otimes\g$, there induces a linear map $T:V\longrightarrow\g$ defined to be
$$\overline T(v,\xi):=\langle\xi,Tv\rangle,\quad\xi\in \g^*,v\in V.$$

The following result demonstrates that a relative Rota-Baxter operator gives rise to a solution to the classical Lie-Yamaguti Yang-Baxter equation in a lager \LYA, which is parallel to the context of Lie algebras or $3$-Lie algebras.
\begin{thm}\label{thm:O}
Let $(\g,\br,,\ltp)$ be a \LYA ~and $(V;\rho,\mu)$ its representation. Then with the above notations, $T:V\longrightarrow\g$ is a relative Rota-Baxter operator on $(\g,\br,,\ltp)$ with respect to $(V;\rho,\mu)$ if and only if $$r=\overline{T}-\sigma_{12}(\overline{T})$$
is an $r$-matrix of the semidirect product \LYA ~$\g\ltimes_{\rho^*,-\mu^*\tau}V^*$.
\end{thm}
\begin{proof}
Let $\{v_1,\cdots,v_n\}$ be a basis for the vector space $V$ and $\{v_1^*,\cdots,v_n^*\}$ its dual basis. Then we have
$$\overline T=\sum_iv_i^*\otimes Tv_i\in V^*\otimes \g\subset \otimes^2\Big(\g\ltimes_{\rho^*,-\mu^*\tau}V^*\Big).$$
By a direct computation, we have
\begin{eqnarray*}
\Courant{r_{13},r_{41},r_{12}}&=&\sum_{ijk}\Big(\Courant{Tv_i,Tv_k,Tv_j}\otimes v_i^*\otimes v_j^*\otimes v_k^*-\Courant{Tv_i,Tv_k,v_j^*}\otimes v_i^*\otimes Tv_j\otimes v_k^*\\
~ &&+\Courant{v_i^*,Tv_k,Tv_j}\otimes Tv_i\otimes v_j^*\otimes v_k^*+\Courant{Tv_i,v_k^*,Tv_j}\otimes v_i^*\otimes v_j^*\otimes Tv_k\Big),\\
\Courant{r_{42},r_{23},r_{21}}&=&\sum_{ijk}\Big(v_j^*\otimes\Courant{Tv_k,Tv_i,Tv_j}\otimes v_i^*\otimes v_k^*-Tv_j\otimes\Courant{Tv_k,Tv_i,v_j^*}\otimes v_i^*\otimes v_k^*\\
~ &&-v_j^*\otimes\Courant{v_k^*,Tv_i,Tv_j}\otimes v_i^*\otimes Tv_k-v_j^*\otimes\Courant{Tv_k,v_i^*,Tv_j}\otimes Tv_i\otimes v_k^*\Big),\\
\Courant{r_{31},r_{32},r_{43}}&=&\sum_{ijk}\Big(v_i^*\otimes v_j^*\otimes \Courant{Tv_i,Tv_j,Tv_k}\otimes v_k^*-v_i^*\otimes v_j^*\otimes \Courant{Tv_i,Tv_j,v_k^*}\otimes Tv_k\\
~ &&-Tv_i\otimes v_j^*\otimes \Courant{v_i^*,Tv_j,Tv_k}\otimes v_k^*-v_i^*\otimes Tv_j\otimes \Courant{Tv_i,v_j^*,Tv_k}\otimes v_k^*\Big),\\
\Courant{r_{41},r_{42},r_{43}}&=&\sum_{ijk}\Big(-v_i^*\otimes v_j^*\otimes v_k^*\otimes\Courant{Tv_i,Tv_j,Tv_k}+v_i^*\otimes v_j^*\otimes Tv_k\otimes \Courant{Tv_i,Tv_j,v_k^*}\\
~ &&-Tv_i\otimes v_j^*\otimes v_k^*\otimes \Courant{v_i^*,Tv_j,Tv_k}+v_i^*\otimes Tv_j\otimes v_k^*\otimes\Courant{Tv_i,v_j^*,Tv_k}\Big).
\end{eqnarray*}
Moreover, we also have that
\begin{eqnarray*}
\sum_i Tv_i\otimes \Courant{Tv_i,Tv_j,v_k^*}&=&\sum_iTv_i\otimes D^*(Tv_i,Tv_j)v_k^*=\sum_iTv_i\otimes \sum_m\langle D^*(Tv_i,Tv_j)v_k^*,v_m\rangle v_m^*\\
~ &=&\sum_{im}Tv_i\otimes \Big(-\langle D(Tv_i,Tv_j)v_m,v_k^*\rangle v_m^*\Big)=-\sum_mT\Big(D(Tv_i,Tv_j)v_m\Big)\otimes v_m^*,
\end{eqnarray*}
and
\begin{eqnarray*}
\sum_i Tv_i\otimes \Courant{v_i^*,Tv_j,Tv_k}&=&\sum_iTv_i\otimes (-\mu^*(Tv_k,Tv_j)v_i^*)=\sum_iTv_i\otimes \sum_m(-\langle \mu^*(Tv_k,Tv_j)v_i^*,v_m\rangle v_m^*)\\
~ &=&\sum_{im}Tv_i\otimes \Big(\langle \mu(Tv_k,Tv_j)v_m,v_i^*\rangle v_m^*\Big)=\sum_mT\Big(\mu(Tv_k,Tv_j)v_m\Big)\otimes v_m^*.
\end{eqnarray*}
Denote by
$$\huaO_1(u,v,w)=\Courant{Tu,Tv,Tw}-T\Big(D(Tu,Tv)w+\mu(Tv,Tw)u-\mu(Tu,Tw)v\Big),\quad \forall u,v,w\in V.$$
Therefore, we have
\begin{eqnarray*}
\llbracket r,r,r\rrbracket&=&\Courant{r_{13},r_{41},r_{12}}+\Courant{r_{42},r_{23},r_{21}}+\Courant{r_{31},r_{32},r_{43}}+\Courant{r_{41},r_{42},r_{43}}\\
~ &=&\sum_{ijk}\Big(\huaO_1(v_i,v_j,v_k)\otimes v_i^*\otimes v_j^*\otimes v_k^*+v_j^*\otimes\huaO_1(v_k,v_i,v_j)\otimes v_i^*\otimes v_k^*\\
~ &&+v_i^*\otimes v_j^*\otimes \huaO_1(v_i,v_j,v_k)\otimes v_k^*-v_i^*\otimes v_j^*\otimes v_k^*\otimes\huaO_1(v_i,v_j,v_k)\Big).
\end{eqnarray*}
Moreover, we also have that
\begin{eqnarray*}
[r,r]&=&[r_{12},r_{13}]+[r_{12},r_{23}]+[r_{13},r_{23}]\\
~ &=&\sum_{ij}\Big(\huaO_2(v_i,v_j)\otimes v_i^*\otimes v_j^*-v_i^*\otimes\huaO_2(v_i,v_j)\otimes v_k^*+v_i^*\otimes v_j^*\otimes \huaO_2(v_i,v_j)\Big),
\end{eqnarray*}
where
$$\huaO_2(u,v):=[Tu,Tv]-T\Big(\rho(Tu)v-\rho(Tv)u\Big),\quad \forall u,v\in V.$$
Hence, $r$ is an $r$-matrix, i.e.,
$$[r,r]=0,\quad{\rm and}\quad \Courant{r,r,r}=0$$
if and only if
$$\huaO_1(v_i,v_j,v_k)=0,\quad{\rm and}\quad \huaO_2(v_i,v_j)=0,$$
for all $i,j,k$, which implies that $T$ is a relative Rota-Baxter operator. This finishes the proof.
\end{proof}

\begin{pro}
Let $(A,*,\{\cdot,\cdot,\cdot\})$ be a pre-\LYA ~and $\{e_i\}_{i=1}^n$ a basis for $A$ and $\{e_i^*\}_{i=1}^n$ its dual basis. Then
$$r:=\sum_{i=1}^n(e_i\otimes e_i^*-e_i^*\otimes e_i)$$
is a skew-symmetric $r$-matrix for the \LYA ~$A\ltimes_{\Ad^*,-R^*\tau}A^*$. Moreover, $r$ is nondegenerate and the induced bilinear form $\huaB$ on $A\ltimes_{\Ad^*,-R^*\tau}A^*$ is given by \eqref{bilinear}.
\end{pro}
\begin{proof}
By Proposition \ref{subad}, we have that the identity map $\Id:A\longrightarrow A$ is a relative Rota-Baxter operator on the sub-adjacent \LYA ~$A^c$ of the given pre-\LYA ~$(A,*,\{\cdot,\cdot,\cdot\})$ with respect to the representation $(A;\Ad,R)$.  Moreover, it follows from Theorem \ref{thm:O} that $r=\sum_{i=1}^n(e_i\otimes e_i^*-e_i^*\otimes e_i)$ is a skew-symmetric solution to the classical Lie-Yamaguti Yang-Baxter equation in $A\ltimes_{\Ad^*,-R^*\tau}A^*$. It is obvious that the corresponding bilinear form $\huaB\in \otimes^2(A\oplus A^*)$ is given by \eqref{bilinear}. The proof is finished.
\end{proof}

In order to generalize a result given by Semonov-Tian-Shansky in \cite{STS} to the context of \LYA s, we need to recall the notion of quadratic \LYA s and prove a lemma first.

\begin{defi}{\rm(\cite{Kikkawa})}
 A {\bf quadratic Lie-Yamaguti algebra} is a Lie-Yamaguti algebra $(\g,[\cdot,\cdot],\Courant{\cdot,\cdot,\cdot})$ equipped with a nondegenerate symmetric bilinear form $\huaB \in \otimes^2\g^*$ satisfying the following invariant conditions
\begin{eqnarray}
\label{invr1}\huaB([x,y],z)&=&-\huaB(y,[x,z]),\\
\label{invr2}\huaB(\Courant{x,y,z},w)&=&\huaB(x,\Courant{w,z,y}), \quad \forall x,y,z \in \g.
\end{eqnarray}
We denote a quadratic \LYA ~by $\Big((\g,\br,,\ltp),\huaB\Big)$.
\end{defi}

\begin{lem}\label{lem:qua}
Let $\Big((\g,\br,,\ltp),\huaB\Big)$ be a quadratic \LYA. Then the induced map $\huaB^\natural:\g\longrightarrow\g^*$ defined by \eqref{Bna} is an isomorphism from the adjoint representation $(\g;\ad,\huaR)$ to the coadjoint representation $(\g^*;\ad^*,-\huaR^*\tau)$.
\end{lem}
\begin{proof}
For all $x,y,z,w\in \g$, we have
\begin{eqnarray*}
\langle\huaB^\natural(\ad_xy)-\ad^*_x\huaB^\natural(y),z\rangle&=&\huaB([x,y],z)+\langle\huaB^\natural(y),[x,z]\rangle\\
~ &=&\huaB([x,y],z)+\huaB(y,[x,z])\\
~ &=&0.
\end{eqnarray*}
Since $z$ is arbitrary, we deduce that
\begin{eqnarray*}
\huaB^\natural(\ad_xy)=\ad^*_x\huaB^\natural(y), \quad \forall x,y\in \g.
\end{eqnarray*}
Similarly, we also have that
\begin{eqnarray*}
\langle\huaB^\natural(\huaR(x,y)z)+\huaR^*(y,x)\huaB^\natural(z),w\rangle&=&\huaB(\Courant{z,x,y},w)+\langle\huaR^*(y,x)\huaB^\natural(z),w\rangle\\
~ &=&\huaB(\Courant{z,x,y},w)-\huaB(z,\Courant{w,y,x})\\
~ &=&0.
\end{eqnarray*}
Since $w$ is arbitrary, we deduce that
\begin{eqnarray*}
\huaB^\natural(\huaR(x,y)z)=-\huaR^*(y,x)\huaB^\natural(z), \quad \forall x,y,z\in \g.
\end{eqnarray*}
Hence, $\huaB^\natural$ is an isomorphism between adjoint representation and coadjoint representation. This completes the proof.
\end{proof}

\begin{cor}
Let $\Big((\g,\br,,\ltp),\huaB\Big)$ be a quadratic \LYA. Then $\huaB^\natural:\g\longrightarrow\g^*$ satisfies
\begin{eqnarray*}
\huaB^\natural(\huaL(x,y)z)=\huaL^*(x,y)\huaB^\natural(z),\quad\forall x,y,z\in \g.
\end{eqnarray*}
\end{cor}
\begin{proof}
The proof is a direct computation and is similar to that of Lemma \ref{lem:qua}.
\end{proof}

It is in a position to generalize the result given by Semonov-Tian-Shansky to the context of \LYA s.
\begin{thm}\label{sts}
Let $\Big((\g,\br,,\ltp),\huaB\Big)$ be a quadratic \LYA ~and $T:\g^*\longrightarrow\g$ a linear map. Then $T$ is a relative Rota-Baxter operator on $(\g,\br,,\ltp)$ with respect to the coadjoint representation $(\g^*;\ad^*,-\huaR^*\tau)$ if and only if $T\circ\huaB^\natural$ is a relative Rota-Baxter operator on $(\g,\br,,\ltp)$ with respect to the adjoint representation $(\g;\ad,\huaR)$.
\end{thm}
\begin{proof}
For all $x,y,z\in \g$, by Lemma \ref{lem:qua}, we have that
\begin{eqnarray*}
(T\circ\huaB^\natural)\Big([T\circ\huaB^\natural(x),y]+[x,T\circ\huaB^\natural(y)]\Big)&=&T\Big(\huaB^\natural(\ad_{T\circ\huaB^\natural(x)}y)
-\huaB^\natural(\ad_{T\circ\huaB^\natural(y)}x)\Big)\\
~ &=&T\Big(\ad^*_{T\circ\huaB^\natural(x)}\huaB^\natural(y)-\ad^*_{T\circ\huaB^\natural(y)}\huaB^\natural(x)\Big),
\end{eqnarray*}
and
\begin{eqnarray*}
~ &&(T\circ\huaB^\natural)\Big(\Courant{T\circ\huaB^\natural(x),T\circ\huaB^\natural(y),z}+\Courant{x,T\circ\huaB^\natural(y),T\circ\huaB^\natural(z)}
-\Courant{y,T\circ\huaB^\natural(x),T\circ\huaB^\natural(z)}\Big)\\
~ &=&T\Big(\huaL^*(T\circ\huaB^\natural(x),T\circ\huaB^\natural(y))\huaB^\natural(z)
-\huaR^*(T\circ\huaB^\natural(z),T\circ\huaB^\natural(y))\huaB^\natural(x)+\huaR^*(T\circ\huaB^\natural(z),T\circ\huaB^\natural(x))\huaB^\natural(y)\Big).
\end{eqnarray*}
Thus we obtain that $T$ is a relative Rota-Baxter operator on $(\g,\br,,\ltp)$ with respect to the coadjoint representation $(\g^*;\ad^*,-\huaR^*\tau)$ if and only if $T\circ\huaB^\natural$ is a relative Rota-Baxter operator on $(\g,\br,,\ltp)$ with respect to the adjoint representation $(\g;\ad,\huaR)$. This finishes the proof.
\end{proof}

Theorem \ref{sts} is a generalized result of Semonov-Tian-Shansky's in \cite{STS} to the context of \LYA s, whereas the generalized result of Leibniz algebra version was given in \cite{T.S2}. The following corollary is directly.

\begin{cor}
Let $\Big((\g,\br,,\ltp),\huaB\Big)$ be a quadratic \LYA. Then $r\in \wedge^2\g$ is a solution to the classical Lie-Yamaguti Yang-Baxter equation in $\g$ if and only if $r^\sharp\circ \huaB^\natural:\g\longrightarrow\g$ is a relative Rota-Baxter operator on $\g$ with respect to the adjoint representation $(\g;\ad,\huaR)$.
\end{cor}

At the end of this section, we introduce the notion of local cocycle Lie-Yamaguti bialgebras.
\begin{defi}\label{local}
A {\bf local cocycle Lie-Yamaguti bialgebra} is a \LYA ~$(\g,\br,,\ltp)$ together with two linear maps $\delta=\delta_1+\delta_2:\g\longrightarrow\otimes^2\g$ and $\omega=\omega_1+\omega_2+\omega_3:\g\longrightarrow\otimes^3\g$ such that $(\delta^*,\omega^*)$ defines a \LYA ~structure on $\g^*$, and the following conditions are satisfied:
\begin{itemize}
\item $\delta_1$ is a $1$-cocycle with respect to the representation $(\otimes^2\g;1\otimes\ad,1\otimes\huaR)$;
\item $\delta_2$ is a $1$-cocycle with respect to the representation $(\otimes^2\g;\ad\otimes1,\huaR\otimes1)$;
\item $\omega_1$ is a $1$-cocycle with respect to the representation $(\otimes^3\g;\ad\otimes1\otimes1,\huaR\otimes1\otimes1)$;
\item $\omega_2$ is a $1$-cocycle with respect to the representation $(\otimes^3\g;1\otimes\ad\otimes1,1\otimes\huaR\otimes1)$;
\item $\omega_3$ is a $1$-cocycle with respect to the representation $(\otimes^3\g;1\otimes1\otimes\ad,1\otimes1\otimes\huaR)$.
\end{itemize}
\end{defi}

\begin{rmk}\label{ltpbi}
When a given \LYA ~$(\g,\br,,\ltp)$ reduces to a Lie triple system $(\g,\ltp)$, we obtain the local cocycle bialgebra structure of a Lie triple system: there exists a coalgebra structure $\omega=\omega_1+\omega_2+\omega_3:\g\longrightarrow\otimes^3\g$ on the Lie triple system $(\g,\ltp)$ such that the following conditions are satisfied:
\begin{itemize}
\item $\omega_1$ is a $1$-cocycle with respect to the representation $(\otimes^3\g;\ad\otimes1\otimes1,\huaR\otimes1\otimes1)$;
\item $\omega_2$ is a $1$-cocycle with respect to the representation $(\otimes^3\g;1\otimes\ad\otimes1,1\otimes\huaR\otimes1)$;
\item $\omega_3$ is a $1$-cocycle with respect to the representation $(\otimes^3\g;1\otimes1\otimes\ad,1\otimes1\otimes\huaR)$,
\end{itemize}
where $(\g;\huaR)$ is the adjoint representation of the Lie triple system $\g$.
\end{rmk}

\begin{rmk}
We would like to point out that $\delta_i$ and $\omega_j$ $(1\leqslant i\leqslant 2,1\leqslant j\leqslant 3)$ as in Eqs. \eqref{c1} and \eqref{c2} are not $1$-cocycles of a \LYA ~$\g$ in general, thus a solution to the classical Lie-Yamaguti Yang-Baxter equation can not give rise to a local cocycle Lie-Yamaguti bialgebra structure.
Unlike $3$-Lie algebras, even for a Lie triple system, these ${\omega_j}$'s are not $1$-cocycles any more, which implies that a solution to the classical Yang-Baxter equation does not
produce a local cocycle bialgebra structure in the context of Lie triple systems. This illustrates that there is a huge difference between $3$-Lie algebras and Lie triple systems.

\end{rmk}

\section{Manin triples, matched pairs, and double construction Lie-Yamaguti bialgebras}
In this section, we consider double construction Lie-Yamaguti bialgebras and clarify the relationship between double construction Lie-Yamaguti bialgebras and local cocycle Lie-Yamaguti bialgebras.
First, we introduce the notion of Manin triples.

\begin{defi}\label{Manin}
Let $\g_1$ and $\g_2$ be two \LYA s. A {\bf Manin triple} of $\g_1$ and $\g_2$ is a quadratic \LYA ~$\Big((\g,\br,,\ltp),\huaB\Big)$ such that
\begin{itemize}
\item[(i)] $\g=\g_1\oplus\g_2$ as vector spaces;
\item[(ii)] $\g_1$ and $\g_2$ are subalgebras of $\g$ which  are isotropic, i.e., $\huaB(x_1,y_1)=\huaB(x_2,y_2)=0$, for any $x_1,y_1\in \g_1$ and $x_2,y_2\in \g_2$;
\item[(iii)] For all $x_1,y_1\in \g_1$ and $x_2,y_2\in \g_2$, we have
\begin{eqnarray*}
\pr_1\Courant{x_1,y_1,x_2}=0,\quad \pr_1\Courant{x_1,x_2,y_1}=0,\quad \pr_2\Courant{x_2,y_2,x_1}=0,\quad \pr_2\Courant{x_2,x_1,y_2}=0,
\end{eqnarray*}
where $\pr_1$ and $\pr_2$ are projections from $\g_1\oplus\g_2$ to $\g_1$ and $\g_2$ respectively.
\end{itemize}
We denote a Manin triple of \LYA s by $\Big((\g,\huaB),\g_1,\g_2\Big)$ or simply by $(\g,\g_1,\g_2)$.
\end{defi}

\begin{rmk}
Recall that a product structure on a \LYA ~$(\g,\br,,\ltp)$ is a Nijenhuis operator $E:\g\longrightarrow\g$ satisfying $E^2=\Id$.
\emptycomment{
\begin{eqnarray*}
[Ex,Ey]&=&E[Ex,y]+E[x,Ey]-[x,y],\\
\Courant{Ex,Ey,Ez}&=&E\Courant{Ex,Ey,z}+E\Courant{x,Ey,Ez}+E\Courant{Ex,y,Ez}\\
~ &&-\Courant{Ex,y,z}-\Courant{x,Ey,z}-E\Courant{x,y,Ez}+E\Courant{x,y,z},\quad\forall x,y,z\in \g.
\end{eqnarray*}}
There exists a product structure $E$ on $\g$ if and only if $\g$ admits a decomposition into two subalgebras: $\g=\g_1\oplus\g_2$. Moreover, Condition (iii) in Definition \ref{Manin} is just the condition that makes the product structure perfect. See \cite{Sheng Zhao} for more details about product structures and complex structures on \LYA s. Thus a Manin triple $\Big((\g,\huaB),\g_1,\g_2\Big)$ of \LYA s is in fact the quadratic \LYA ~$(\g,\huaB)$ such that there is a perfect product structure on $\g$ whose decomposed subalgebras are isotropic.
\end{rmk}

\smallskip
Let $\Big((\g,\huaB),\g_1,\g_2\Big)$ and $\Big((\g',\huaB'),\g_1',\g_2'\Big)$ be two Manin triples of \LYA s. An {\bf isomorphism} between $\Big((\g,\huaB),\g_1,\g_2\Big)$ and $\Big((\g',\huaB'),\g_1',\g_2'\Big)$ is an isomorphism between \LYA s $f:\g\longrightarrow\g'$ such that
$$f(\g_1)\subset\g_1',\quad f(\g_2)\subset\g_2',\quad \huaB(x,y)=\huaB'(f(x),f(y)),\quad\forall x,y\in \g$$

\smallskip
Let $(\g,\br,,\ltp)$ and $(\g^*,\br_{_*},\ltp_*)$ be a \LYA s. There is a natural nondegenerate symmetric bilinear form $\huaB$ on $\g\oplus\g^*$ given by
\begin{eqnarray}
\huaB(x+\xi,y+\eta)=\langle x,\eta\rangle+\langle \xi,y\rangle,\quad \forall x,y\in \g,~\xi,\eta\in \g^*.\label{bilinear}
\end{eqnarray}

Define a pair of operations $([\cdot,\cdot]_{\g\oplus\g^*},\Courant{\cdot,\cdot,\cdot}_{\g\oplus\g^*})$ to be
\begin{eqnarray}
\label{double1}[x+\xi,y+\eta]_{\g\oplus\g^*}&=&[x,y]+\ad_x^*\xi-\ad_y^*\eta\\
~\nonumber &&+[\xi,\eta]_*+\mathfrak{ad}_\xi^*y-\mathfrak{ad}_\eta^*x,\\
\label{double2}\Courant{x+\xi,y+\eta,z+\zeta}_{\g\oplus\g^*}&=&\Courant{x,y,z}+\huaL^*(x,y)\zeta-\huaR^*(z,y)\xi+\huaR^*(z,x)\eta\\
~\nonumber &&+\Courant{\xi,\eta,\zeta}_*+\frkL^*(\xi,\eta)z-\frkR^*(\zeta,\eta)x+\frkR^*(\zeta,\xi)y,
\end{eqnarray}
for all $x,y,z\in \g$ and $\xi,\eta,\zeta\in \g^*$. Here $(\ad^*,-\huaR^*\tau)$ and $(\mathfrak{ad}^*,-\frkR^*\tau)$ are the coadjoint representations of $\g$ on $\g^*$ and $\g^*$ on $\g$ respectively, where $\huaL^*=D_{\ad^*,-\huaR^*\tau}$ and $\frkL^*=D_{\mathfrak{ad}^*,-\frkR^*\tau}$.

Note that the bracket $([\cdot,\cdot]_{\bowtie},\Courant{\cdot,\cdot,\cdot}_{\bowtie})$ given by \eqref{double1} and \eqref{double2} is invariant with respect to the bilinear form $\huaB$ given by \eqref{bilinear} and satisfies the Condition (iii) in Definition \ref{Manin}. If $(\g\oplus\g^*,\br_{_{\g\oplus\g^*}},\ltp_{\g\oplus\g^*})$ is a \LYA, then it is easy to see that $\g$ and $\g^*$ are isotropic subalgebras with respect to the bilinear form $\huaB$ given by \eqref{bilinear}. Consequently, $\Big((\g\oplus\g^*,\huaB),\g,\g^*\Big)$ is a Manin triple of $\g$ and $\g^*$, which is called the {\bf standard Manin  triple}.

\begin{pro}\label{stand}
Any Manin triple of \LYA s is isomorphic to a standard one.
\end{pro}
\begin{proof}
Let $\g_1$ and $\g_2$ be \LYA s. If $\Big((\g=\g_1\oplus\g_2,\huaB),\g_1,\g_2\Big)$ is a Manin triple of $\g_1$ and $\g_2$, then $\g_2$ is isomorphic to $\g_1^*$ as vector spaces via
$$\langle \alpha,x\rangle:=\huaB(\alpha,x),\quad\forall \alpha\in \g_2,~x\in \g_1.$$
Moreover, $\g_1^*$ is equipped with a \LYA  ~structure from $\g_2$ via this isomorphism. Then $\Big((\g_1\oplus\g_2,\huaB),\g_1,\g_2\Big)$ is isomorphic to the standard Manin triple $\Big((\g_1\oplus\g_1^*,\huaB),\g_1,\g_1^*\Big)$. This completes the proof.
\end{proof}

\begin{rmk}
By the proof of Proposition \ref{stand}, we obtain that any Manin triple of \LYA s $\Big((\g,\huaB),\g_1,\g_2\Big)$ is also isomorphic to the standard Manin triple $\Big((\g_2^*\oplus\g_2,\huaB),\g_2^*,\g_2\Big)$ of $\g_2$ and $\g_2^*$. So the statement of the proposition is that any Manin triple of \LYA s is isomorphic to ``a" standard one, not ``the".
\end{rmk}

In the following, let us introduce the notion of matched pairs of \LYA s. Let $(\g_1,\br_{_1},\ltp_1)$ and $(\g_2,\br_{_2},\ltp_2)$ be \LYA s~ and $\rho_1:\g_1\longrightarrow\gl(\g_2),~\mu_1:\otimes^2\g_1\longrightarrow\gl(\g_2)$ and $\rho_2:\g_2\longrightarrow\gl(\g_1),~\mu_2:\otimes^2\g_2\longrightarrow\gl(\g_1)$ be linear maps. Define a pair of linear brackets $(\br_{_{\bowtie}},\ltp_{\bowtie})$ on $\g_1\oplus\g_2$ to be
\begin{eqnarray}
\label{MT1}[x+u,y+v]_{\bowtie}&=&[x,y]_1+\rho_2(u)y-\rho_2(v)x\\
~\nonumber &&+[u,v]_2+\rho_1(x)v-\rho_1(y)u,\\
\label{MT2}\Courant{x+u,y+v,z+w}_{\bowtie}&=&\Courant{x,y,z}_1+D_2(u,v)z+\mu_2(v,w)x-\mu_2(u,w)y\\
~\nonumber &&+\Courant{u,v,w}_2+D_1(x,y)w+\mu_2(y,z)u-\mu_2(x,z)v,
\end{eqnarray}
for all $x,y,z\in \g_1,~u,v,w\in \g_2$, where $D_1:=D_{\rho_1,\mu_1}$ and $D_2:=D_{\rho_2,\mu_2}$. Note that in general the bracket operation $(\br_{_{\bowtie}},\ltp_{\bowtie})$ need not satisfy the conditions of \LYA s.

\begin{rmk}
Note that the operation $(\br_{_{\g\oplus\g^*}},\ltp_{\g\oplus\g^*})$ defined by \eqref{double1} and \eqref{double2} is a special case for $(\br_{_{\bowtie}},\ltp_{\bowtie})$ defined by \eqref{MT1} and \eqref{MT2}, where $\g_1=\g,~\g^*=\g_2$, and $\rho_1=\ad^*,~\mu_1=-\huaR^*\tau,~\rho_2=\frkad^*,~\mu_2=-\frkR^*\tau$.
\end{rmk}

\begin{defi}
Let $(\g_1,\br_{_1},\ltp_1)$ and $(\g_2,\br_{_2},\ltp_2)$ be two \LYA s. If
the operation $(\br_{_{\bowtie}},\ltp_{\bowtie})$ defined by \eqref{MT1} and \eqref{MT2} forms a \LYA ~structure on $\g_1\oplus\g_2$, then  we say that a quadruple $\Big(\g_1,\g_2;(\rho_1,\mu_1),(\rho_2,\mu_2)\Big)$ is a {\bf matched pair} of \LYA s.
\end{defi}

\begin{pro}\label{pro:MT}
With the above notations, the quadruple $\Big(\g_1,\g_2;(\rho_1,\mu_1),(\rho_2,\mu_2)\Big)$ is a matched pair of \LYA s if and only if the following conditions hold
\begin{itemize}
\item[\rm(i)] $(\g_2;\rho_1,\mu_1)$ is a representation of $\g_1$;
\item[\rm (ii)] $(\g_1;\rho_2,\mu_2)$ is a representation of $\g_2$;
\item[\rm (iii)] the following equalities hold:
\begin{eqnarray}
\label{mtp1}~&&[\rho_2(u)x,y]_1-\rho_2(\rho_1(x)u)y-\rho_2(u)[x,y]_1-[\rho_2(u)y,x]_1+\rho_2(\rho_1(y)u)x=0,\\
~\label{mtp2} &&\Courant{\rho_2(u)x,y,z}_1=\Courant{\rho_2(u)y,x,z}_1\\
~ \label{mtp3}&&\mu_2(u,v)[x,y]_1-\mu_2(\rho_1(y)u,v)x+\mu_2(\rho_1(x)u,v)y=0,\\
~ \label{mtp4}&&\Courant{x,y,\rho_2(u)z}_1=\rho_2(D_1(x,y)u)z+\rho_2(u)\Courant{x,y,z}_1,\\
~ \label{mtp5}&&\mu_2(u,\rho_1(x)v)y=[x,\mu_2(u,v)y]_1,\\
~\label{mtp6} &&\rho_2(\mu_1(x,y)u)z=\rho_2(\mu_1(x,z)u)y,\\
~ &&\Courant{x,y,\mu_2(u,v)z}_1=\mu_2(u,v)\Courant{x,y,z}_1+\mu_2(D_1(x,y)u,v)z+\mu_2(u,D_1(x,y)v)z,\\
~\label{mtp8} &&\mu_2(u,\mu_1(x,y)v)z=\Courant{\mu_2(u,v)z,x,y}_1-D_2(v,\mu_1(z,x)u)y+\mu_2(v,\mu_1(z,y)u)x,\\
~\label{mtp10} &&\mu_2(u,\mu_2(x,y)v)z=D_2(\mu_1(z,x)u,v)y-\Courant{x,\mu_2(u,v)z,y}_1+\mu_2(v,\mu_2(z,y)u)x,\\
~\label{mtp9}&&[\rho_1(x)u,v]_1-\rho_1(\rho_2(u)x)v-\rho_1(x)[u,v]_2-[\rho_1(x)v,u]_2+\rho_1(\rho_2(v)x)u=0,\\
~ &&\Courant{\rho_1(x)u,v,w}_2=\Courant{\rho_1(x)v,u,w}_2\\
~ &&\mu_1(x,y)[u,v]_2-\mu_1(\rho_2(v)x,y)u+\mu_1(\rho_2(u)x,y)v=0,\\
~ &&\Courant{u,v,\rho_1(x)w}_2=\rho_1(D_2(u,v)x)w+\rho_1(x)\Courant{u,v,w}_2,\\
~ &&\mu_1(x,\rho_2(u)y)v=[u,\mu_1(x,y)v]_2,\\
~ &&\rho_1(\mu_2(u,v)x)w=\rho_1(\mu_2(u,w)x)v,\\
~ &&\Courant{u,v,\mu_1(x,y)w}_2=\mu_1(x,y)\Courant{u,v,w}_2+\mu_1(D_2(u,v)x,y)w+\mu_1(x,D_2(u,v)y)w,\\
~\label{mtp16} &&\mu_1(x,\mu_2(u,v)y)w=\Courant{\mu_1(x,y)w,u,v}_2-D_1(y,\mu_2(w,u)x)v+\mu_1(y,\mu_2(w,v)x)u,\\
~\label{mtp18} &&\mu_1(x,\mu_1(u,v)y)w=D_1(\mu_2(w,u)x,y)v-\Courant{u,\mu_1(x,y)w,v}_2+\mu_1(y,\mu_1(w,v)x)u,
\end{eqnarray}
\end{itemize}
for all $x,y,z\in \g_1$ and $u,v,w\in \g_2$. Here, $D_1=D_{\rho_1,\mu_1}$ and $D_2=D_{\rho_2,\mu_2}$.
\end{pro}
\begin{proof}
It is a direct computation, so we omit the details.
\end{proof}

A direct computation leads to the following corollary.
\begin{cor}
With the assumptions in Proposition \ref{pro:MT}, we have the following equalities:
\begin{eqnarray*}
~ D_2(\rho_1(x)u,v)&=&D_2(\rho_1(x)v,u),\\
~ D_2(u,v)[x,y]_1&=&[D_2(u,v)x,y]_1+[x,D_2(u,v)y]_1,\\
~ D_2(u,v)\Courant{x,y,z}_1&=&\Courant{D_2(u,v)x,y,z}_1+\Courant{x,D_2(u,v)y,z}_1+\Courant{x,y,D_2(u,v)z}_1,\\
~ D_1(\rho_2(u)x,y)&=&D_1(\rho_2(u)y,x),\\
~ D_2(x,y)[u,v]_2&=&[D_2(x,y)u,v]_2+[u,D_2(x,y)v]_2,\\
~ D_2(x,y)\Courant{u,v,w}_2&=&\Courant{D_2(x,y)u,v,w}_2+\Courant{u,D_2(x,y)v,w}_2+\Courant{u,v,D_2(x,y)w}_2,
\end{eqnarray*}
for all $x,y,z\in \g_1$ and $u,v,w\in \g_2$.
\end{cor}

The following proposition reveals the relationship between matched pairs and Manin triples of \LYA s.
\begin{pro}\label{thm:mt}
Let $(\g,\br,,\ltp)$ and $(\g^*,\br_{_*},\ltp_*)$ be \LYA s. Then the quadruple $\Big(\g,\g^*;(\ad^*,-\huaR^*\tau),(\mathfrak{ad}^*,-\frkR^*\tau)\Big)$ is a matched pair of $\g$ and $\g^*$ if and only if the triple $\Big((\g\oplus\g^*,\huaB),\g,\g^*\Big)$ is a Manin triple of $\g$ and $\g^*$, where the invariant bilinear form $\huaB$ is given by Eq. \eqref{bilinear}.
\end{pro}
\begin{proof}
Let $\Big(\g,\g^*;(\ad^*,-\huaR^*\tau),(\mathfrak{ad}^*,-\frkR^*\tau)\Big)$ be  a matched pair of \LYA s. Then $(\g\oplus\g^*,[\cdot,\cdot]_{\g\oplus\g^*},\Courant{\cdot,\cdot,\cdot}_{\g\oplus\g^*})$ is a \LYA, where $([\cdot,\cdot]_{\g\oplus\g^*},\Courant{\cdot,\cdot,\cdot}_{\g\oplus\g^*})$ is given by \eqref{double1} and \eqref{double2}. We only need to prove that $\huaB$ satisfies the invariant condition \eqref{invr1} and \eqref{invr2}. Indeed, for all $x,y,z,w\in \g$ and $\xi,\eta,\zeta,\delta\in \g^*$, we have
\begin{eqnarray*}
~ &&\huaB(x+\xi,[y+\eta,z+\zeta]_{\g\oplus\g^*})\\
~ &=&\huaB(x+\xi,[y,z]+\ad_y^*\zeta-\ad_z^*\eta+[\eta,\zeta]_*+\frkad_\eta^*z-\frkad_\zeta^*y)\\
~ &=&\pair{x,\ad_y^*\zeta-\ad_z^*\eta+[\eta,\zeta]_*}+\pair{\xi,[y,z]+\frkad_\eta^*z-\frkad_\zeta^*y}\\
~ &=&-\pair{[y,x],\zeta}+\pair{[z,x],\eta}+\pair{x,[\eta,\zeta]_*}\\
~ &&+\pair{\xi,[y,z]}-\pair{[\eta,\xi]_*,z}+\pair{[\zeta,\xi]_*,y},
\end{eqnarray*}
on the other hand, we also have that
\begin{eqnarray*}
~ &&\huaB([x+\xi,y+\eta]_{\g\oplus\g^*},z+\zeta)\\
~ &=&\huaB([x,y]+\ad_x^*\eta-\ad_y^*\xi+[\xi,\eta]_*+\frkad_\xi^*y-\frkad_\eta^*x,z+\zeta)\\
~ &=&\pair{\ad_x^*\eta-\ad_y^*\xi+[\xi,\eta]_*,z}+\pair{[x,y]+\frkad_\xi^*y-\frkad_\eta^*x,\zeta}\\
~ &=&-\pair{\eta,[x,z]}+\pair{\xi,[y,z]}+\pair{[\xi,\eta]_*,z}\\
~ &&+\pair{[x,y],\zeta}-\pair{y,[\xi,\zeta]_*}+\pair{x,[\eta,\zeta]_*},
\end{eqnarray*}
which implies that
$$\huaB(x+\xi,[y+\eta,z+\zeta]_{\g\oplus\g^*})=\huaB([x+\xi,y+\eta]_{\g\oplus\g^*},z+\zeta).$$
Moreover, we have
\begin{eqnarray*}
~ &&\huaB(\Courant{x+\xi,y+\eta,z+\zeta}_{\g\oplus\g^*},w+\delta)\\
~ &=&\huaB\Big(\Courant{x,y,z}+\huaL^*(x,y)\zeta-\huaR^*(z,y)\xi+\huaR^*(z,x)\eta\\
~ &&+\Courant{\xi,\eta,\zeta}_*+\frkL^*(\xi,\eta)z-\frkR^*(\zeta,\eta)x
+\frkR^*(\zeta,\xi)y,w+\delta\Big)\\
~ &=&\pair{\Courant{x,y,z}+\frkL^*(\xi,\eta)z-\frkR^*(\zeta,\eta)x+\frkR^*(\zeta,\xi)y,\delta}\\
~ &&+\pair{\Courant{\xi,\eta,\zeta}_*+\huaL^*(x,y)\zeta-\huaR^*(z,y)\xi+\huaR^*(z,x)\eta,w}\\
~ &=&\pair{\Courant{x,y,z},\delta}-\pair{z,\Courant{\xi,\eta,\delta}_*}+\pair{x,\Courant{\delta,\zeta,\eta}_*}-\pair{y,\Courant{\delta,\zeta,\xi}_*}\\
~ &&+\pair{\Courant{\xi,\eta,\zeta}_*,w}-\pair{\zeta,\Courant{x,y,w}_*}+\pair{\xi,\Courant{w,z,y}}-\pair{\eta,\Courant{w,z,x}},
\end{eqnarray*}
on the other hand, we also have that
\begin{eqnarray*}
~ &&\huaB(x+\xi,\Courant{w+\delta,z+\zeta,y+\eta}_{\g\oplus\g^*})\\
~ &=&\huaB\Big(x+\xi,\Courant{w,z,y}+\huaL^*(w,z)\eta-\huaR^*(y,z)\delta+\huaR^*(y,w)\zeta\\
~ &&+\Courant{\delta,\zeta,\eta}_*+\frkL^*(\delta,\zeta)y-\frkR^*(\eta,\zeta)w+\frkR^*(\eta,\delta)z\Big)\\
~ &=&\pair{x,\Courant{\delta,\zeta,\eta}_*+\huaL^*(w,z)\eta-\huaR^*(y,z)\delta+\huaR^*(y,w)\zeta}\\
~ &&+\pair{\xi,\Courant{w,z,y}+\frkL^*(\delta,\zeta)y-\frkR^*(\eta,\zeta)w+\frkR^*(\eta,\delta)z}\\
~ &=&\pair{x,\Courant{\delta,\zeta,\eta}_*}-\pair{\Courant{w,z,x},\eta}+\pair{\Courant{x,y,z},\delta}-\pair{\Courant{x,y,w},\zeta}\\
~ &&+\pair{\xi,\Courant{w,z,y}}-\pair{\Courant{\delta,\zeta,\xi}_*,y}+\pair{\Courant{\xi,\eta,\zeta}_*,w}-\pair{\Courant{\xi,\eta,\delta}_*,z},
\end{eqnarray*}
which implies that
$$\huaB(\Courant{x+\xi,y+\eta,z+\zeta}_{\g\oplus\g^*},w+\delta)=\huaB(x+\xi,\Courant{w+\delta,z+\zeta,y+\eta}_{\g\oplus\g^*}).$$

Conversely, if $\Big((\g\oplus\g^*,\huaB),\g,\g^*\Big)$ is a Manin triple, where $\huaB$ is an invariant bilinear form given by \eqref{bilinear}. For all $x\in \g$ and $\xi,\eta,\zeta\in \g^*$, by \eqref{invr1}, we have
\begin{eqnarray*}
\pair{\eta,\rho_2(\xi)x}=\huaB(\eta,[\xi,x]_{\g\oplus\g^*})=-\huaB([\xi,\eta]_*,x)=-\pair{[\xi,\eta]_*,x}=\langle \eta,\frkad_\xi^*x\rangle,
\end{eqnarray*}
which implies that $\rho_2=\frkad^*.$ Moreover, by \eqref{invr2}, we also have
\begin{eqnarray*}
\pair{\zeta,\mu_2(\xi,\eta)x}=\huaB(\zeta,\Courant{x,\xi,\eta}_{\g\oplus\g^*})=\huaB(\Courant{\zeta,\eta,\xi}_*,x)=\pair{\Courant{\zeta,\eta,\xi}_*,x}
=-\pair{\zeta,\frkR^*(\eta,\xi)x},
\end{eqnarray*}
which implies that $\mu_2=-\frkR^*\tau.$ Similarly, we have that $\rho_1=\ad^*$ and $\mu_1=-\huaR^*\tau$. Thus we obtain that $\Big(\g,\g^*,(\ad^*,-\huaR^*\tau),(\frkad^*,-\frkR^*\tau)\Big)$ is a matched pair of \LYA s. This completes the proof.
\end{proof}

It is in a position to introduce the notion of double construction Lie-Yamaguti bialgebras. Before this, we show the following proposition.
\begin{pro}\label{LYbi}
Let $(\g,\br,,\ltp)$ be a \LYA ~endowed with two linear maps $\delta:\g\longrightarrow\otimes^2\g$ and $\omega:\g\longrightarrow\otimes^3\g$. Then $\Big(\g,\g^*;(\ad^*,-\huaR^*\tau),(\frkad^*,-\frkR^*\tau)\Big)$ is a matched pair of $(\g,\br,,\ltp)$ and $(\g^*,\br_{_*},\ltp_*)$ if and only if the following conditions are satisfied
\begin{itemize}
\item[\rm(i)] $(\g,\delta,\omega)$ is a Lie-Yamaguti coalgebra;
\item[\rm(ii)] the following compatibility conditions are satisfied: $\forall x,y,z\in \g,$
\begin{eqnarray}
~\label{LYbi1} \delta([x,y])&=&\Big(\ad_x\otimes1+1\otimes\ad_x\Big)\delta(y)-\Big(\ad_y\otimes1+1\otimes\ad_y\Big)\delta(x),\\
~\label{LYbi2}\Big(1\otimes \huaR(y,z)\Big)\delta(x)&=&\Big(1\otimes \huaR(x,z)\Big)\delta(y),\\
~\label{LYbi3} \omega([x,y])&=&\Big(1\otimes1\otimes\ad_x\Big)\omega(y)-\Big(1\otimes1\otimes\ad_y\Big)\omega(x),\\
~\delta(\Courant{x,y,z})&=&\Big(\huaL(x,y)\otimes 1+1\otimes\huaL(x,y)\Big)\omega(z),\\
~\Big(\huaR(y,z)\otimes1\Big)\delta(x)&=&\Big(\huaR(x,z)\otimes1\Big)\delta(y),\\
~\omega(\Courant{x,y,z})&=&\Big(\huaL(x,y)\otimes1\otimes1+1\otimes\huaL(x,y)\otimes1+1\otimes1\otimes\huaL(x,y)\Big)\omega(z),
\end{eqnarray}
\begin{eqnarray}
~\label{LYbi7} \Big(1\otimes \huaR(y,x)\otimes1-R(x,y)\otimes1\otimes1\Big)\omega(z)&=&\sigma_{12}\sigma_{23}\Big(1\otimes\huaR(x,z)\otimes1\Big)\omega(y)\\
~\nonumber&&+\sigma_{23}\Big(1\otimes\huaR(y,z)\otimes1\Big)\omega(x).
\end{eqnarray}
\end{itemize}
\end{pro}
\begin{proof}
It is sufficient to show that Eqs. \eqref{mtp1}-\eqref{mtp18} are equivalent to Condition (ii). Note that when $\rho_1=\ad^*, ~\mu_1=-\huaR^*\tau$ and $\rho_2=\frkad^*,~\mu_2=-\frkR^*\tau$, Eqs. \eqref{mtp1}-\eqref{mtp4} and Eqs. \eqref{mtp6}-\eqref{mtp8} are equivalent to Conditions \eqref{LYbi1}-\eqref{LYbi7} respectively, Eq. \eqref{mtp5} is equivalent to that $\omega$ is skew-symmetric with respect to the first two variables, and moreover Eq. \eqref{mtp8} and \eqref{mtp10} are equivalent.
Indeed, for all $x,y,z\in \g$ and $\xi,\eta\in \g^*$, let us now compute that
\begin{eqnarray*}
~ &&\pair{[\frkad_\xi^*x,y]-\frkad_{\ad_x^*\xi}^*y-\frkad_\xi^*[x,y]-[\frkad_\xi^*y,x]+\frkad_{\ad_y^*\xi}^*x,\eta}\\
~ &=&-\pair{\delta(x),\xi\otimes\ad_y^*\eta}+\pair{\delta(y),\ad_x^*\xi\otimes\eta}+\pair{\delta([x,y]),\xi\otimes\eta}+\pair{\delta(y),\xi\otimes\ad_x^*\eta}
-\pair{\delta(x),\ad_y^*\xi\otimes\eta}\\
~ &=&\pair{\delta([x,y])-(\ad_x\otimes1+1\otimes\ad_x)\delta(y)+(\ad_y\otimes1+1\otimes\ad_y)\delta(x),\xi\otimes\eta},
\end{eqnarray*}
which implies that Eq. \eqref{mtp1} is equivalent to Eq. \eqref{LYbi1}. Moreover, we have that
\begin{eqnarray*}
\pair{\Courant{\frkad_\xi^*x,y,z}-\Courant{\frkad_\xi^*y,x,z},\zeta}&=&-\pair{\frkad_\xi^*x,\huaR^*(y,z)\eta}+\pair{\frkad_\xi^*y,\huaR^*(x,z)\eta}\\
~ &=&\pair{\delta(x),\xi\otimes\huaR^*(y,z)\eta}-\pair{\delta(y),\xi\otimes\huaR^*(x,z)\eta}\\
~ &=&\pair{-\Big(1\otimes\huaR(y,z)\Big)\delta(x)+\Big(1\otimes\huaR(x,z)\Big)\delta(y),\xi\otimes\eta},
\end{eqnarray*}
which implies that Eq. \eqref{mtp2} is equivalent to Eq. \eqref{LYbi2}. Similarly, we obtain that Eqs. \eqref{mtp3}-\eqref{mtp4} and Eqs. \eqref{mtp6}-\eqref{mtp8} are equivalent to Eqs. \eqref{LYbi3}-\eqref{LYbi7}. What is left is to show that Eqs. \eqref{mtp9}-\eqref{mtp16} are equivalent to Eqs. \eqref{mtp1}-\eqref{mtp8} respectively. We only prove the equivalence of Eq. \eqref{mtp8} and Eq. \eqref{mtp16} since others are similar. Indeed, we have that
\begin{eqnarray*}
~ &&\pair{\huaR^*(\frkR^*(\eta,\xi)y,x)\zeta+\Courant{\huaR^*(y,x)\zeta,\xi,\eta}_*-\huaL^*(y,\frkR^*(\xi,\zeta)x)\eta-\huaR^*(\frkR^*(\eta,\zeta)x,y)\xi,z}\\
~ &=&\pair{\zeta,\Courant{\frkR^*(\eta,\xi)y,z,x}}+\pair{\zeta,\huaR(y,x)\frkR^*(\xi,\eta)z}-\pair{\eta,\huaR(y,z)\frkR^*(\xi,\zeta)x}-\pair{\xi,\huaR(z,y)\frkR^*(\eta,\zeta)x}\\
~ &=&-\pair{\Courant{\eta,\huaR^*(z,x)\zeta,\xi}_*,y}-\pair{\Courant{\xi,\huaR^*(y,x)\zeta,\eta}_*,z}-\pair{\Courant{\huaR^*(y,z)\eta,\xi,\zeta}_*,x}
+\pair{\Courant{\eta,\huaR^*(z,y)\xi,\zeta}_*,x}\\
~ &=&\langle(\huaR(y,z)\otimes1\otimes1-1\otimes\huaR(z,y)\otimes1)\omega(x)+\sigma_{12}\sigma_{23}(1\otimes\huaR(y,x)\otimes1)\omega(z)+\sigma_{23}
(1\otimes\huaR(z,x)\otimes1)\omega(y),\\
~ && \eta\otimes\xi\otimes\zeta\rangle,
\end{eqnarray*}
which gives the equivalence of Eqs. \eqref{mtp8} and \eqref{mtp16}. This completes the proof.
\end{proof}

\begin{defi}
Let $(\g,\br,,\ltp)$ be a \LYA, and structure maps $\delta:\g\longrightarrow\otimes^2\g$ and $\omega:\g\longrightarrow\otimes^3\g$ linear maps. If Conditions (i) and (ii) in Proposition \ref{LYbi} are satisfied, then we say that $\g$ is a {\bf double construction Lie-Yamaguti bialgebra}. We denote a Lie-Yamaguti bialgebra by $(\g,\br,,\ltp,\delta,\omega)$, or simply by $(\g,\g^*)$.
\end{defi}

The following corollary is obvious.
\begin{cor}
If $(\g,\g^*)$ is a double construction Lie-Yamaguti bialgebra, then so is $(\g^*,\g)$.
\end{cor}

By Proposition \ref{thm:mt} and Proposition \ref{LYbi}, we obtain the following theorem directly.
\begin{thm}\label{main}
Let $(\g,\br,,\ltp)$ be a \LYA ~and $\delta:\g\longrightarrow\otimes^2\g$ and $\omega:\g\longrightarrow\otimes^3\g$ linear maps. Suppose that the structure map $(\delta^*,\omega^*)$ defines a \LYA ~structure on $\g^*$. Then the following statements are equivalent:
\begin{itemize}
\item[\rm(1)] the \LYA ~$\g$ makes $(\g,\g^*)$ into a double construction Lie-Yamaguti bialgebra;
\item[\rm(2)] the quadruple $\Big(\g,\g^*;(\ad^*,-\huaR^*\tau),(\mathfrak{ad}^*,-\frkR^*\tau)\Big)$ is a matched pair of \LYA s;
\item[\rm(3)] the triple $\Big((\g\oplus\g^*,\huaB),\g,\g^*\Big)$ is a standard Manin triple, where the invariant bilinear form $\huaB$ is given by \eqref{bilinear}.
\end{itemize}
In this case, the \LYA ~$(\g\oplus\g^*,\br_{_{\g\oplus\g^*}},\ltp_{\g\oplus\g^*})$ is called the {\bf double} of the Lie-Yamaguti bialgebra $(\g,\g^*)$, and is denoted by $\g\bowtie\g^*$.
\end{thm}

\emptycomment{
\begin{rmk}
Since there does not exist the tensor representation in the context of \LYA s, we are not able to study Lie-Yamaguti bialgebras by coboundary methods. In the context of $3$-Lie algebra
$(\g,[\cdot,\cdot,\cdot])$, we have $(\ot^3\g;\huaL\otimes 1\otimes1)$, $(\ot^3\g;1\otimes\huaL\otimes1)$, and $(\ot^3\g;1\otimes1\otimes\huaL)$ are representations of the $3$-Lie algebra $\g$, hence authors studied local cocycle $3$-Lie bialgebras in \cite{B.G.S}. This property even fails in the context of \LYA s. Thus we cannot study the local cocycle Lie-Yamaguti bialgebras parallel to the $3$-Lie bialgebras, leading to not being able to explore the relationship between the classical Yang-Baxter equation and Lie-Yamaguti bialgebras. We, however, clarify the relationship between the solution to the classical Yang-Baxter equation in \LYA s and relative Rota-Baxter operators in the next section. And one sees that our bialgebra theory for \LYA s enjoys many good properties parallel to that for Lie algebras in the sequel, which justifies its correctness.
\end{rmk}
}

The following proposition reveals the relationship between local cocycle Lie-Yamaguti bialgebras and double construction Lie-Yamaguti bialgebras.
\begin{pro}
A double construction Lie-Yamaguti bialgebra gives rise to a local cocycle Lie-Yamaguti bialgebra.
\end{pro}
\begin{proof}
Let $(\g,\delta,\omega)$ be a double construction Lie-Yamaguti bialgebra. Let $k_1,k_2,k_3$ be complex numbers such  that $k_1=k_2$ and $k_1+k_2+k_3=1$. Denote by $\delta_i=\half\delta$ and $\omega_j=k_j\omega$, where $i=1,2;~j=1,2,3$. Set $\delta=\delta_1+\delta_2$ and $\omega=\omega_1+\omega_2+\omega_3$. It is obvious that $\delta_1, \delta_2$ are $1$-cocycles of $\g$ with respect to the representations $(\otimes^2\g;1\otimes\ad,1\otimes\huaR),(\otimes^2\g;\ad\otimes1,\huaR\otimes1)$ respectively, and $\omega_1,\omega_2,\omega_3$ are $1$-cocycles of $\g$ with respect to $(\otimes^3\g;\ad\otimes1\otimes1,\huaR\otimes1\otimes1),
(\otimes^3\g;1\otimes\ad\otimes1,1\otimes\huaR\otimes1),~(\otimes^3\g;1\otimes1\otimes\ad,1\otimes1\otimes\huaR)$ respectively. Hence $(\delta,\omega)$ defines a local cocycle Lie-Yamaguti bialgebra on $\g$.
\end{proof}

We give some examples of double construction Lie-Yamaguti bialgebras to end up with this section. As a first example, we have the following trivial Lie-Yamaguti bialgebra.
\begin{ex}
For any \LYA ~$\g$, taking $\delta=0$ and $\omega=0$, then $(\g,\delta,\omega)$ is a Lie-Yamaguti bialgebra. In this case, the corresponding Manin triple gives a quadratic \LYA ~$(\g\ltimes_{\ad^*,-\huaR^*\tau}\g^*,\huaB)$. Dually, for any trivial \LYA ~$\g$ (that is, both binary and trinary brackets are zero), any \LYA ~structure $(\delta^*,\omega^*)$ on the dual space $\g^*$ makes $(\g,\delta,\omega)$ a Lie-Yamaguti bialgebra. Such Lie-Yamaguti bialgebra is called the {\bf trivial Lie-Yamaguti bialgebra}.
\end{ex}

\begin{ex}
Let $\g$ be the $2$-dimensional \LYA ~given in Example \ref{ex:2di}.
The nonzero cobrackets  $\delta:\g\longrightarrow\otimes^2\g$ and $\omega:\g\longrightarrow\otimes^3\g$ are given by
$$\delta(e_1)=e_1\otimes e_2,\quad\omega(e_1)=e_1\otimes e_2\otimes e_2.$$
Then $\delta^*:\otimes^2\g^*\longrightarrow\g^*$ and $\omega^*:\otimes^3\g^*\longrightarrow\g^*$ defines a pair of \LYA ~structure on $\g^*$ that is isomorphic to $\g$. It is direct to see that $(\g,\delta,\omega)$ is a Lie-Yamaguti bialgebra.
\end{ex}

\begin{ex}
Let $\g$ be a $4$-dimensional \LYA ~with a basis $\{e_1,e_2,e_3,e_4\}$ defined to be
$$[e_1,e_2]=2e_4,\quad\Courant{e_1,e_2,e_1}=e_4.$$
If the dual of linear maps $\delta:\g\longrightarrow\otimes^2\g$ and $\omega:\g\longrightarrow\otimes^3\g$  forms a \LYA~structure on $\g^*$ such that $(\g\oplus\g^*,\g,\g^*)$ is a Manin triple, where the invariant bilinear form is given by \eqref{bilinear}. Then $\delta=0$ and $\omega=0$.
\end{ex}


\end{document}